\documentclass[12pt]{article}

\usepackage{mathrsfs}

\usepackage{amsmath,amsthm,amssymb}
\usepackage{graphicx}
\usepackage{vector}
\usepackage{pdfsync}
\usepackage{psfrag}
\usepackage{appendix}

\font\tencmmib=cmmib10 \skewchar\tencmmib '60
\newfam\cmmibfam
\textfont\cmmibfam=\tencmmib

\def\lessim{\ \lower4pt\hbox{$
\buildrel{\displaystyle <}\over\sim$}\ }
\def\gessim{\ \lower4pt\hbox{$\buildrel{\displaystyle >}
\over\sim$}\ }

\newcommand{\e}{\mathbb{E}}
\newcommand{\R}{\mathbb{R}}
\newcommand{\p}{\mathbb{P}}

\newcommand{\vsi}{{\boldsymbol{\sigma}}}

\newtheorem{lemma}{\bf Lemma}

\newtheorem{theorem}{\bf Theorem}

\newtheorem{example}{\bf Example}

\newtheorem{proposition}{\bf Proposition}

\theoremstyle{remark}

\newenvironment{Proof of lemma}{\noindent{\bf Proof of Lemma}}{\hfill$\Box$\newline}
\newenvironment{Proof of theorem}{\noindent{\bf Proof of Theorem}}{\hfill{\footnotesize${\square}$}\newline}
\newenvironment{Proof of theorems}{\noindent{\bf Proof of Theorems}}{\hfill$\Box$\newline}
\newenvironment{Proof of proposition}{\noindent{\bf Proof of Proposition}}{\hfill$\Box$\newline}
\newenvironment{Proof of propositions}{\noindent{\bf Proof of Propositions}}{\hfill$\Box$\newline}
\newenvironment{Proof of exercise}{\noindent{\it Proof of Exercise:}}{\hfill$\Box$}
\begin{document}

\title{On properties of Parisi measures}

\author{Antonio Auffinger  \thanks{auffing@math.uchicago.edu} \\ \small{University of Chicago}\and Wei-Kuo Chen \thanks{wkchen@math.uchicago.edu} \\ \small{University of Chicago} }

\maketitle

\begin{abstract} We investigate the structure of Parisi measures, the functional order parameters of mixed $p$-spin models in mean field spin glasses. In the absence of external field, we prove that a Parisi measure satisfies the following properties. First, at all temperatures, the support of any Parisi measure contains the origin. If it contains an open interval, then the measure has a smooth density on this interval. Next, we give a criterion on temperature parameters for which a Parisi measure is neither Replica Symmetric nor One Replica Symmetry Breaking. Finally, we show that in the Sherrington-Kirkpatrick model, slightly above the critical temperature, the largest number in the support of a Parisi measure is a jump discontinuity. An analogue of these results is discussed in the spherical mixed $p$-spin models. As a tool to establish these facts and of independent interest, we study functionals of the associated Parisi PDEs and derive regularity properties of their solutions.
\end{abstract}

%{Keywords: Parisi measure}

\section{Introduction and main results}\label{intro}

The mixed $p$-spin model is one of the most fundamental mean field spin glasses. The study of this model has provided a rich collection of problems and phenomena both in the physical and mathematical sciences. The reader interested in the background, history and methodologies is invited to check the books of Mezard-Parisi-Virasoro \cite{MPV}, Talagrand \cite{Tal11} and the numerous references therein. 

In this paper we are interested in the structure of the functional order parameter of this model in the absence of external field. This order parameter, also known as the Parisi measure, is predicted to fully qualitatively describe the system and has been the main subject of study by several authors both in physics and mathematics \cite{MPV, Tal07}. Although the role of the order parameter has been partially unveiled and significant progress has been made in the recent years, the structure of the Parisi measures remains very mysterious at low temperature.

Let us now describe the mixed $p$-spin model. For $N\geq 1,$ let $\Sigma_N:=\{-1,+1\}^N$ be the Ising spin configuration space. Consider the pure $p$-spin Hamiltonian with $p\geq 2$,
\begin{equation}
  \label{eq:Hamiltonianpspin}
  H_{N,p}(\boldsymbol\sigma) = \frac{1}{N^{(p-1)/2}}
  \sum_{i_1, \dots, i_p=1}^N
  g_{i_1, \dots, i_p} \sigma_{i_1}\dots \sigma_{i_p}
\end{equation}
for $\vsi=(\sigma_1,\ldots,\sigma_N)\in\Sigma_N,$ where the random variables $g_{i_1, \dots, i_p}$ are independent standard Gaussian for all $p\geq 2$ and $(i_1,\ldots,i_p).$ The mixed $p$-spin model is defined on $\Sigma_N$ and its Hamiltonian is given by a linear combination of the pure $p$-spin Hamiltonians,
\begin{equation}\label{eq:bqs}
H_{N}(\vsi) = \sum_{p=2}^{\infty} \beta_p  H_{N,p}(\vsi).
\end{equation}
Here the sequence $\boldsymbol \beta:=(\beta_p)_{p\geq 2}$ is called the temperature parameters and satisfies $\sum_{p=2}^\infty 2^p\beta_p^2<\infty$ that is enough to guarantee the well-definedness of the model. It is easy to compute that the covariance of $H_N$ is given by 
$$
\e H_{N}(\vsi)H_N(\vsi')=N\xi(R(\vsi,\vsi')),
$$
where 
\begin{align*}
R(\vsi,\vsi'):=\frac{1}{N}\sum_{i=1}^N\sigma_i\sigma_i'
\end{align*} 
is the overlap between spin configurations $\vsi$ and $\vsi'$ and 
\begin{align*}
\xi(u):=\sum_{p=2}^\infty\beta_p^2u^p.
\end{align*}
When $\xi(u)=\beta_2^2 u^2$, we recover the famous Sherrigton-Kirkpatrick model \cite{SK75}. The Gibbs measure is defined as 
$$ 
G_N(\vsi) = \frac{\exp H_{N}(\vsi)}{Z_N},\,\,\forall \vsi\in\Sigma_N,
$$
where the normalizing factor $Z_N$ is known as the partition function.
The central goal and most important problem in this model is to understand the large $N$ behavior of these measures at different values of $\boldsymbol \beta$. This is intimately related to the computation of the free energy $N^{-1}\log Z_N$ in the thermodynamic limit and, as a result, has been studied extensively since the ground-breaking work of G. Parisi \cite{Par79, Par80}. 

In the Parisi solution, it was predicted that the thermodynamic limit of the free energy can be computed by a variational formula. More precisely, consider the Parisi functional $\mathcal{P}$ (see \eqref{eq:ParFunctionalFormula}) defined on the space $M_d[0,1]$ of all probability measures on $[0,1]$ consisting of a finite number of atoms. Then the following limit exists almost surely,
\begin{equation}\label{ParisiFormula}
\lim_{N\rightarrow \infty} \frac{1}{N} \log Z_N = \inf_{\mu\in M_d[0,1]} \mathcal P(\mu).
\end{equation}
For the detailed mathematical proof of this result, the readers are referred to $\cite{Pan11:1,Tal06}.$ It is known \cite{Guerra03} that the Parisi functional can be extended continuously to the space $M[0,1]$ of all probability measures on $[0,1]$ with respect to the metric $d(\mu,\mu'):=\int_0^1|\mu([0,u])-\mu'([0,u])|du$. This guarantees that the infinite dimensional variational problem on the left side of \eqref{ParisiFormula} always has a minimizer. Throughout the paper, we will call any such minimizer  a {\it Parisi measure} and denote it by $\mu_P$. It is expected that for any mixed $p$-spin model, the Parisi measure is unique and it gives the limit law of the overlap $R(\vsi^1,\vsi^2)$ under $\e G_N^{\otimes 2}.$ Ultimately, it fully describes the limit of replicas $(\vsi^\ell)_{\ell\geq 1}$ with respect to the measure $\e G_N^{\otimes\infty}.$ Under certain assumptions on the temperature parameters, these have been rigorous verified in recent years, see \cite{Pan12:1} for an overview along this direction, but the general case remains open. 

The main objective of this paper is to establish some qualitative properties about Parisi measures that have been predicted in physics literature. We now summarize these predictions. Denote by $\mbox{supp}\,\mu_P$ the support of $\mu_P$ and by $q_{M}$ the largest number in $\mbox{supp}\,\mu_P.$ We say that a Parisi measure is Replica Symmetric (RS) if it is a Dirac measure; One Replica Symmetric Breaking (1RSB) if it consists of two atoms; Full Replica Symmetric Breaking (FRSB) if it contains a continuous component on some interval contained in its support.
For the Sherrington-Kirkpatrick model $(\xi(u)=\beta_2^2u^2)$ with no external field, if $0<\beta_2<1/\sqrt{2}$, the model is RS: $\mu_P=\delta_0$. (This region of temperature, known as the high temperature region, is different from the familiar one $\beta_2<1$ in the original SK model \cite{SK75}, because our Hamiltonian sums over all $1\leq i_1,i_2\leq N$). In the low temperature regime, $\beta_2>1/\sqrt{2}$, the model exhibits FRSB behavior: $\mu_P=\nu+(1-m)\delta_{q_{M}}.$ Here $\nu$ is a fully supported measure on $[0,q_{M}]$ with $m:=\nu([0,q_{M}])<1$ and possesses a smooth density. For detailed discussion, the readers are referred to Chapter III in \cite{MPV}.

\begin{figure*}[h]
\centering
\scalebox{1.1}{\includegraphics[width=0.70\textwidth]{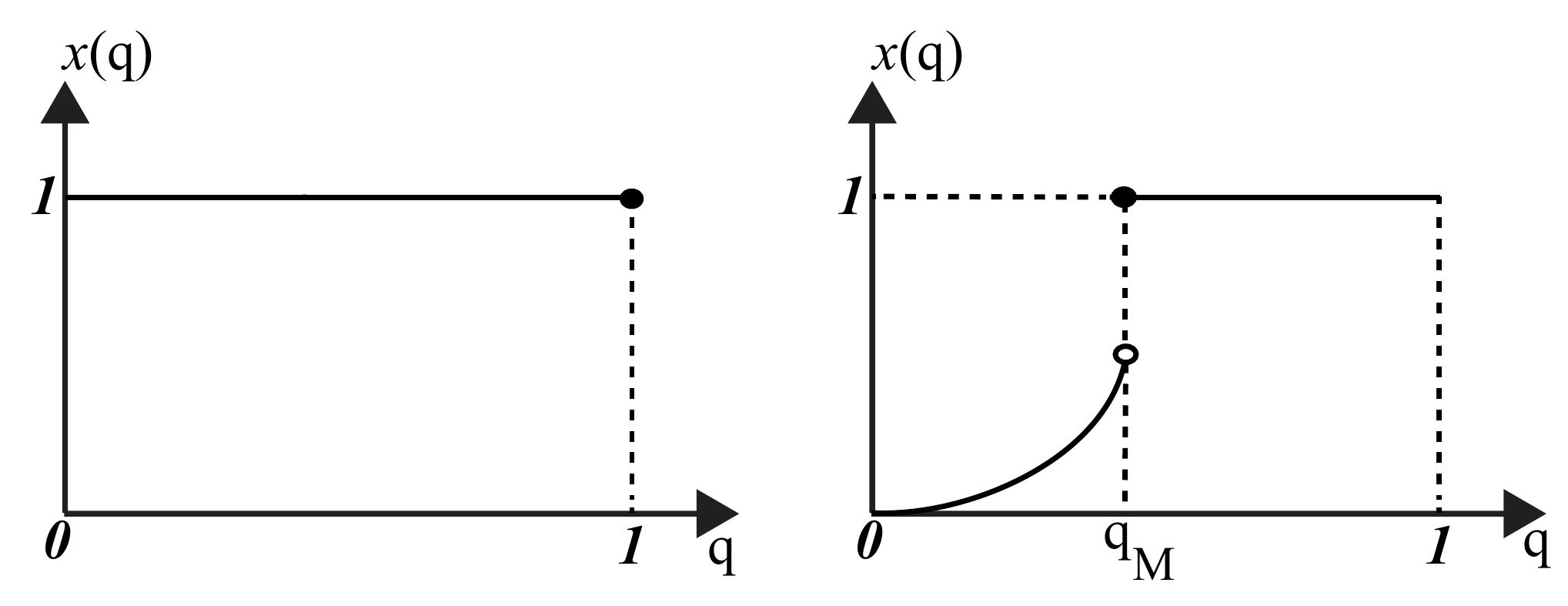}}
\caption{Schematic forms of the order parameter $x(q)=\mu_P([0,q])$ for the Sherrington-Kirkpatrick model at zero magnetic field \cite[Page 41]{MPV}. The left picture is the order parameter in RS phase, while the right is in FRSB phase.}
\label{fig:SK1}
\end{figure*}

In the case of pure $p$-spin model with $p\geq 3$ $(\xi(u)=\beta_p^2u^p),$ it is conjectured in the work of Gardner \cite{Gardner} that the model in the absence of external field goes through two phase-transitions described by two critical temperatures $\beta_{p,c_1}$ and $\beta_{p,c_2}.$ First at high temperature $\beta_p<\beta_{p,c_1}$, the model is RS: $\mu_P=\delta_0.$ In the low temperature region $\beta_{p,c_1} <\beta_p < \beta_{p,c_2}$, the model is 1RSB: $\mu_P=m\delta_0+(1-m)\delta_{q_{M}}$ for $0<m<1$. Last, at very low temperature $\beta_p > \beta_{p,c_2}$, the Parisi measure is FRSB: $\mu_P=m\delta_0+\nu+(1-m')\delta_{q_{M}},$ where $\nu$ is a fully supported measure on $[q,q_{M}]$ for some $q>0$ with $m':=\nu([q,q_{M}])+m$ and has a smooth density.

\begin{figure}[h]
\label{fig:SK2}
\centering
\scalebox{1.4}{\includegraphics[width=0.71\textwidth]{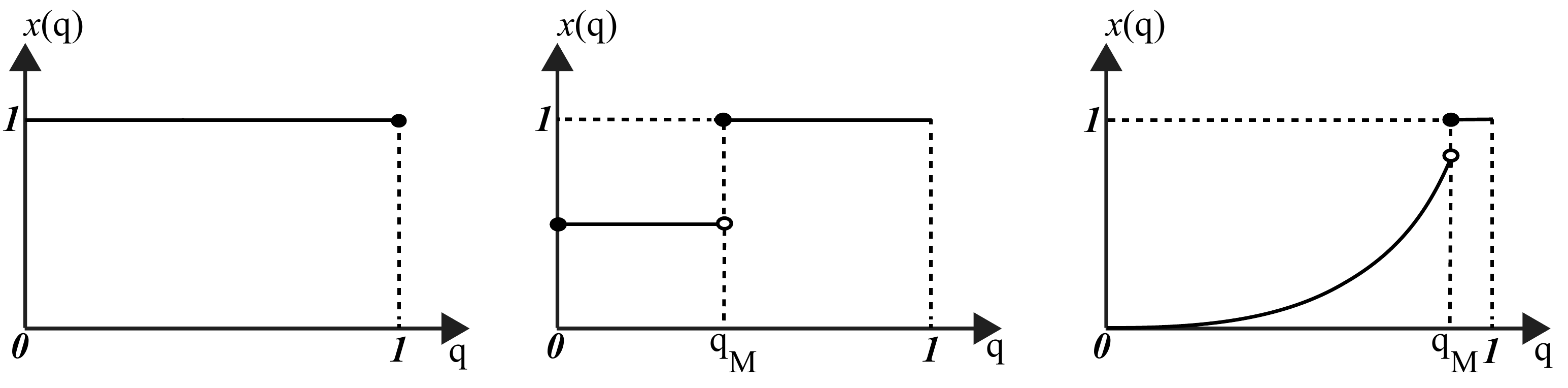}}
\caption{Schematic forms of the order parameter $x(q)=\mu_P([0,q])$ for the pure $p$-spin model with $p\geq 3$ at zero magnetic field \cite{Gardner}. The pictures from left to right are order parameter in RS phase, 1RSB phase and FRSB phase, respectively.}
\end{figure}

To the best of our knowledge, the preceding discussions are by far the most well-known predictions about the structures of Parisi measures studied in physics literature. For mixed models, the behavior may be slightly different as one may expect more phase transitions.  Examples of different structures of Parisi measures were also obtained in the spherical models  \cite{Pan+Tal07, Tal00}. To sum up what we have already discussed up to now, all these models in the absence of external field share three general phenomena: 

\begin{itemize}
\item[$(P1)$]  The origin is contained in the support of the Parisi measure at any temperature.
\item[$(P2)$]  One expects FRSB behavior at low temperature.
\item[$(P3)$]  Any Parisi measure has a jump discontinuity at $q_{M}$ at any temperature.
\end{itemize}

Our main results about these predictions are stated as follows. The first one establishes $(P1)$ and provides a condition on $\xi$ that determines when $0$ is an isolated point of the support.

\begin{theorem}\label{thm:theorem1}
Let $\mu_P$ be a Parisi measure. Then we have that
\begin{itemize}
\item[$(i)$] $0\in\mbox{supp }\mu_P.$
\item[$(ii)$] If $\xi''(0)<1,$ then $\mu_P([0,\hat{q}))=\mu_P(\{0\})$, where $\hat{q}\in [0,1]$ satisfies $\xi''(\hat{q})=1.$
\end{itemize}
\end{theorem}

The next two theorems go in the direction of $(P2).$ We start by establishing two results on the regularity of a Parisi measure. The first one shows that a Parisi measure cannot have a jump at point of accumulation from both sides of the support. The second states that if a Parisi measure is not purely atomic then it must have a smooth density.
\begin{theorem}\label{thm:continuity}
\label{thm3} Let $\mu_P$ be a Parisi measure. 

\begin{enumerate}
\item[$(i)$]  Suppose that there exist an increasing sequence $(u_\ell^-)_{\ell\geq 1}$ and a decreasing sequence $(u_\ell^+)_{\ell\geq 1}$ of $\mbox{supp }\mu_P$ such that $\lim_{\ell\rightarrow\infty}u_\ell^-= u_0=\lim_{\ell\rightarrow\infty}u_\ell^+.$ Then $\mu_P$ is continuous at $u_0.$ 
\item[$(ii)$]  If $(a,b)\subset \mbox{supp }\mu_P$ for some $0\leq a<b\leq 1,$ then the distribution function of $\mu_P$ is infinitely differentiable on $[a,b).$
\end{enumerate}
\end{theorem}

Recall that we say a Parisi measure is RS if it is a Diract measure and is 1RSB if it consists of two atoms. In what follows, we show that a Parisi measure has more complicated structure at very low temperature.

\begin{theorem}\label{sec3:thm1}
Suppose that $\xi$ satisfies 
\begin{align}
\label{sec3:thm1:eq1}
\xi(1)&>\max\left(8\log 2,\frac{1}{3}\sqrt{\xi'(1)}2^{\frac{\xi'(1)}{\xi(1)}+5}\right).
\end{align}
Then the Parisi measure $\mu_P$ is neither RS nor 1RSB.
\end{theorem}

In other words, the criterion \eqref{sec3:thm1:eq1} ensures that the support of a Parisi measure contains at least three points. This is the first result that provides a partial evidence toward the conjecture $(P2).$ Below we list two examples, the pure $p$-spin model and the $(2+p)$-spin model, for which the condition \eqref{sec3:thm1:eq1} can be easily simplified.

\begin{example}[pure $p$-spin model]\rm 
Recall the pure $p$-spin has $\xi(u)=\beta_p^2u^p$. Condition $(\ref{sec3:thm1:eq1})$ on $\xi$ is equivalent to $\beta_p>\max(2\sqrt{2\log 2},  2^{p+5}\sqrt{p}/3).$
\end{example}

\begin{example}[$(2+p)$-spin model]\rm The Hamiltonian of the $(2+p)$-spin model is governed by $\xi(u)=\beta_2^2u^2+\beta_p^2u^{p}$ for $p\geq 3.$ If $\beta_2$ and $\beta_p$ satisfy 
\begin{align}
\label{ex2:eq1}\xi(1)\geq \frac{1}{9}p2^{2p+10}
\end{align} 
then this model is neither RS nor 1RSB. Indeed, if $\beta_2$ and $\beta_p$ satisfy \eqref{ex2:eq1}, it implies $\xi(1)>8\log 2.$ On the other hand, since $\xi'(1)/\xi(1)=(2\beta_2^2+p\beta_p^2)/(\beta_2^2+\beta_p^2)\leq p$, this and \eqref{ex2:eq1} imply
\begin{align*}
\xi(1)\geq \frac{1}{3}\sqrt{p\xi(1)}2^{p+5}\geq \frac{1}{3}\sqrt{\xi'(1)}2^{\frac{\xi'(1)}{\xi(1)}+5}.
\end{align*}
Therefore, \eqref{sec3:thm1:eq1} is satisfied.
\end{example}

As we have mentioned before, in the case of the SK model, the Parisi measure is RS in the high temperature regime $\beta_2<1/\sqrt{2}$. This was proved by Aizenman, Lebowitz and Ruelle in \cite{ALR87}. Later it is also understood by Toninelli in \cite{T02} that a Parisi measure is not RS in the low temperature region $\beta_2>1/\sqrt{2}$. Note that, as we discussed before, as long as $\beta_2$ is above the critical temperature $1/\sqrt{2}$, the SK model is conjectured to be FRSB. In the following, we prove that $(P3)$ holds if the SK temperature $\beta_2$ is slight above the critical temperature $1/\sqrt{2}$ and the total effect of the rest of the mixed $p$-spin interactions with $p\geq 3$ is sufficiently small.

\begin{theorem}\label{thm:top}
Suppose that $\xi$ satisfies $1/\sqrt{2}<\beta_2\leq 3/2\sqrt{2}$ and
\begin{align}
\label{thm:top:eq1}
\frac{\xi^{'''}(1)}{6}+\frac{2}{3}\sqrt{\xi''(1)}\leq 1.
\end{align}
Then the Parisi measure $\mu_P$ has a jump discontinuity at $q_{M}$.
\end{theorem}

 The fact that a $q_{M}$ is a jump discontinuity of the Parisi measure was one of the main assumptions in Theorem 15.4.4 in \cite{Tal11} to prove a decomposition of the system in pure states. The theorem above provides the first non-trivial example where this hypothesis is satisfied. 

\begin{example}[the SK model]\rm Consider the SK model $\xi(u)=\beta_2^2u^2.$ A direct computation yields that $\xi''(1)=2\beta_2^2$ and $\xi'''(1)=0$. If $1/\sqrt{2}<\beta_2\leq 3/2\sqrt{2},$ then \eqref{thm:top:eq1} is satisfied and thus $q_{M}$ is a jump discontinuity of the Parisi measure. 
\end{example}

The rest of the paper is organized as follows. In the next section, we introduce the Parisi PDE and investigate its regularity. We then study the behavior of the Parisi functional near Parisi measures in Section \ref{section:parisifun}. The results therein are the main tools used in Section \ref{Section:s} where we prove Theorems 1, 2, 3 and 4. In Section \ref{section:spherical} we discuss analogues of our results in the spherical $p$-spin model. We end the paper with an Appendix that discuss uniform convergence of derivatives of solutions of the Parisi PDE.

\section{The Parisi PDE and its properties} \label{section:pde}

We now define the Parisi functional and PDE. As in the previous section, we denote by $M_d[0,1]$ the collection of all probability measures on $[0,1]$ consisting of a finite number of atoms and by $M[0,1]$ the collection of all probability measure on $[0,1]$. Each $\mu\in M_d[0,1]$ uniquely  corresponds to a triplet $(k,\mathbf{m},\mathbf{q})$ in such a way that $\mu([0,q_p])=m_p$ for $0\leq p\leq k+1$, where $k\geq 0,$ $\mathbf{m}=(m_p)_{0\leq p\leq k+1}$ and $\mathbf{q}=(q_p)_{0\leq p\leq k+2}$ satisfy
\begin{align}
\begin{split}\label{lem1:eq0}
m_0&=0\leq m_1<m_2<\cdots<m_k\leq m_{k+1}=1,\\
q_0&=0\leq q_1<q_2<\cdots<q_{k+1}\leq q_{k+2}=1.
\end{split}
\end{align}

The Parisi functional $\mathcal P$ is introduced as follows. Consider independent centered Gaussian random variables $(z_j)_{0\leq j \leq k+1}$ with variances $\e z_p^2=\xi'(q_{p+1})-\xi'(q_p)$. Starting from 
$$
X_{k+2}=\log \cosh \sum_{p=0}^{k+1} z_p,
$$
we define recursively for $0\leq p\leq k+1$,
\begin{equation}\label{eq:ParisiRecursion}
X_p=\frac{1}{m_p}\log \e_p \exp m_p X_{p+1},
\end{equation}
where $\e_p$ denotes expectation in the random variables $(z_j)_{j\geq p}$. When $m_p=0$, this means $X_p=\e_pX_{p+1}$. The Parisi functional for $\mu$ is defined as
\begin{equation}\label{eq:ParFunctionalFormula}
\mathcal P(\mu) = X_0 - \frac{1}{2}\int_0^1 u\xi''(u)\mu([0,u])du.
 \end{equation}
 
One may alternatively represent this functional by using the Parisi PDE. Let $\Phi_\mu$ be the solution to the following nonlinear antiparabolic PDE,
\begin{align}
\label{eq:ParPDE}
\partial_u \Phi_\mu(x,u) = - \frac{\xi''(u)}{2}\left(\partial_{x}^2\Phi_\mu (x,u) +\mu([0,u]) \ (\partial_x \Phi_\mu(x,u))^2\right),\,\,(x,u)\in\mathbb{R}\times[0,1]
\end{align}
with end condition $\Phi_\mu(x,1)=\log\cosh x.$ Since the distribution function of $\mu$ is a step function, such equation can be explicitly solved by using the Cole-Hopf transformation. Indeed, for $q_{k+1}\leq u\leq 1,$ 
\begin{align}
\label{eq:ParPDE:eq1}
\Phi_\mu(x,u)=\log \cosh x+\frac{1}{2}(\xi'(1)-\xi'(u))
\end{align}
and one can solve the equation decreasingly to get that for $q_p\leq u<q_{p+1}$ with $0\leq p\leq k,$
\begin{align}
\label{eq:ParPDE:eq2}
\Phi_\mu(x,u)=\frac{1}{m_p}\log \e \exp m_p\Phi_\mu(x+z\sqrt{\xi'(q_{p+1})-\xi'(u)},q_{p+1})
\end{align}
where $z$ is a standard Gaussian random variable.
Now $\Phi_\mu$ and $(X_p)_{0\leq p\leq k+2}$ are related through $\Phi_\mu(\sum_{p'=0}^{p-1}z_{p'},q_{p})=X_p$ for $1\leq p\leq k+2$ and $\Phi_\mu(0,0)=X_0.$ It is well-known \cite{Guerra03} that $\mu\mapsto\Phi_\mu$ defines a Lipschitz functional from $(M_d[0,1],d)$ to $(C(\mathbb{R}\times[0,1]),\|\cdot\|_\infty)$. We can then extend this mapping continuously on $(M[0,1],d)$ and will call $\Phi_\mu$ the Parisi PDE solution associated to $\mu$ for any $\mu\in M[0,1]$. Consequently, this induces a continuous extension of the Parisi functional \eqref{eq:ParFunctionalFormula} to the space $(M[0,1],d).$

Let us proceed to state our main results on some basic properties of the Parisi PDE solutions. Let $B=(B(t))_{t\geq 1}$ be a standard Brownain motion and consider the time changed Brownian motion $M(u)=B({\xi'(u)})$ for $0\leq u\leq 1.$ For any $\mu \in M[0,1],$ we define 
\begin{equation}\label{eq:eqU}
W_\mu(u)=\int_0^u(\Phi_\mu(M(u),u)-\Phi_\mu(M(t),t))d\mu(t),\,\,u\in[0,1].
\end{equation} 
The following two propositions will play an essential role throughout the paper. The first one concerns with the regularity and the uniform convergence of the solutions.

\begin{proposition}\label{thm1} Let $\mu\in M[0,1]$. Suppose that $(\mu_n)_{n\geq 1}\subset M[0,1]$ converges to $\mu.$ 
\begin{itemize}
\item[$(i)$] For $j\geq 0,$ $\partial_x^j\Phi_\mu$ exists and is continuous. Uniformly on $\R \times [0,1]$,
\begin{align*}
\lim_{n\rightarrow\infty}\partial_{x}^j\Phi_{\mu_n}=\partial_x^j\Phi_\mu.
\end{align*}
\item[$(ii)$] Let $P$ be a polynomial on $\mathbb{R}^j$ for some $j\geq 1.$ Then 
\begin{align*}
u&\mapsto \e P(\partial_x\Phi_{\mu}(M(u),u),\ldots,\partial_x^j\Phi_{\mu}(M(u),u))\exp W_{\mu}(u),\,\,u\in[0,1].
\end{align*}
is a continuous function and uniformly on $[0,1]$,
\begin{align*}
&\lim_{n\rightarrow\infty}\e P(\partial_x\Phi_{\mu_n}(M(u),u),\ldots,\partial_x^j\Phi_{\mu_n}(M(u),u))\exp W_{\mu_n}(u)\\
&=\e P(\partial_x\Phi_{\mu}(M(u),u),\ldots,\partial_x^j\Phi_{\mu}(M(u),u))\exp W_{\mu}(u).
\end{align*}
\item[$(iii)$] If $\mu$ is continuous on $[0,1]$, $\partial_u\partial_x^j\Phi_\mu$ is continuous for all $j\geq 0.$ 
\end{itemize}
\end{proposition}

Since the proof of this proposition requires some tedious and technical computations and estimates, we will defer it to the Appendix. Now, we address the behavior of the first and second partial derivatives of the solution with respect to the $x$ variable as well as a property about $W_\mu$. 

\begin{proposition}\label{prop4} Let $\mu\in M[0,1].$ We have that for all $(x,u)\in\mathbb{R}\times[0,1],$
\begin{align}
\begin{split}\label{prop4:eq1}
&\partial_x\Phi_{\mu}(x,-u)=-\partial_x\Phi_{\mu}(x,u),
\end{split}
\\
\begin{split}\label{prop4:eq2}
&|\partial_x\Phi_{\mu}(x,u)|\leq 1,
\end{split}\\
\begin{split}\label{prop4:eq3}
&\frac{C}{\cosh^2 x}\leq \partial_x^2\Phi_{\mu}(x,u)\leq 1,
\end{split}\\
\begin{split}\label{prop4:eq4}
&\e\exp W_{\mu}(u)=1,
\end{split}
\end{align}
where $C>0$ is a constant depending only on $\xi.$
\end{proposition}

\begin{proof}
For $\mu\in M_d[0,1]$, the assertions \eqref{prop4:eq1}, \eqref{prop4:eq2} and \eqref{prop4:eq3} are exactly in Lemma 14.7.16 \cite{Tal11}, while \eqref{prop4:eq4} is valid from $(14.23)$ in \cite{Tal11}. For general $\mu\in M[0,1],$ an approximation argument and $(i)$ in Proposition \ref{thm1} conclude the proof.
\end{proof}

% % % % % % % % % % % % % % % % % % % % % % % % % % % % % % % % % % % % % % % % % % % % % % % % % % % % % % % % % % % % % % % % % % % % % % % % % % % % %
% % % % % % % % % % % % % % % % % % % % % % % % % % % % % % % % % % % % % % % % % % % % % % % % % % % % % % % % % % % % % % % % % % % % % % % % % % % % %
% % % % % % % % % % % % % % % % % % % % % % % % % % % % % % % % % % % % % % % % % % % % % % % % % % % % % % % % % % % % % % % % % % % % % % % % % % % % %

\section{The Parisi functional near Parisi measures}\label{section:parisifun}

Our main approach for understanding Parisi measures is to investigate the Parisi functional around these minimizers. To attain this purpose, we define for any $\mu\in M[0,1],$ 
\begin{align}
\label{sec3:gamma}
\Gamma_\mu(u)=\e (\partial_x\Phi_\mu(M(u),u))^2\exp W_\mu(u),\,\,u\in[0,1].
\end{align}
Suppose that $a$ is a continuous function on $[0,1]$ satisfying $0\leq u+a(u)\leq 1$ for $u\in[0,1]$ and $|a(u)-a(u')|\leq |u-u'|$ for $u,u'\in[0,1].$ For $t\in[0,1]$, let $\mu_t$ be the probability measure induced by the mapping $u\mapsto u+ta(u),$ that is, $\mu_t([0,u+ta(u)])=\mu([0,u])$ for $u\in[0,1].$ It is well-known from Lemma 3.7 \cite{Tal07} that a nontrivial application of the Gaussian integration by parts gives
\begin{align}\label{diff:eq1}
\left.\frac{d}{dt}\mathcal{P}(\mu_t)\right|_{t=0}&=\frac{1}{2}\int_0^1\xi''(u)\left(u-\Gamma_\mu(u)\right)a(u)d\mu(u),
\end{align}
where the left side of \eqref{diff:eq1} is the right derivative. Our main results regarding some basic properties of $\Gamma_\mu$ are summarized below.

\begin{proposition}\label{prop3} 
$\Gamma_\mu$ is differentiable and $\Gamma_\mu'$ is continuous with
\begin{align}
\label{prop3:eq1}
\Gamma_\mu'(u)&=\xi''(u)\e (\partial_x^2\Phi_\mu(M(u),u))^2\exp W_\mu(u).
\end{align}
We have
\begin{align}
\label{prop3:eq2}
\lim_{h\rightarrow 0+}\frac{\Gamma_\mu'(u+h)-\Gamma_\mu'(u)}{h}&= \gamma_{1,\mu}(u)-\mu([0,u])\gamma_{2,\mu}(u)
\end{align}
and
\begin{align}
\begin{split}\label{prop3:eq3}
\lim_{h\rightarrow 0-}\frac{\Gamma_\mu'(u+h)-\Gamma_\mu'(u)}{h}&=\gamma_{1,\mu}(u)-\mu([0,u))\gamma_{2,\mu}(u),
\end{split}
\end{align}
where 
\begin{align*}
\gamma_{1,\mu}(u)&=\xi'''(u)\e (\partial_x^2\Phi_\mu(M(u),u))^2\exp W_\mu(u)\\
&+\xi''(u)^2\e (\partial_{x}^3\Phi_\mu(M(u),u))^2\exp W_\mu(u)
\end{align*}
and 
\begin{align*}
\gamma_{2,\mu}(u)&=2\xi''(u)^2\e(\partial_x^2\Phi_\mu(M(u),u))^3\exp W_\mu(u).
\end{align*}
\end{proposition}

The proof of this proposition will be postponed to the end of this section. In the case that $\mu$ is a Parisi measure,  the left side of \eqref{diff:eq1} is nonnegative. This fact combined with Proposition \ref{prop3} allows us to derive further properties on the first and second derivatives of $\Gamma_{\mu}$ that are stated in the following theorem.

\begin{theorem}\label{thm0}
Let $\mu_P$ be a Parisi measure. Then $\Gamma_{\mu_P}(u)=u$ and $\Gamma_{\mu_P}'(u)\leq 1$ for all $u\in \mbox{supp }\mu_P.$
\end{theorem}

\begin{proof}
The assertion $\Gamma_{\mu_P}(u)=u$ for $u\in \mbox{supp }\mu_P$ has firstly appeared in  \cite[Proposition 3.2]{Tal07}. It can be simply argued as follows. Observe that since $\mu_P$ minimizes the Parisi functional,  \eqref{diff:eq1} gives 
\begin{align}
\label{thm0:proof:eq2}
\frac{1}{2}\int_0^1\xi''(u)\left(u-\Gamma_{\mu_P}(u)\right)a(u)d\mu_P(u)=\left.\frac{d}{dt}\mathcal{P}({\mu_{P,t}})\right|_{t=0}\geq 0
\end{align}
for arbitrary choice of $a$ satisfying $0\leq u+a(u)\leq 1$ for $u\in[0,1]$ and $|a(u)-a(u')|\leq |u-u'|$ for $u,u'\in[0,1],$ where $\mu_{P,t}$ is induced by the mapping $u\mapsto u+a(u)t$ and $\mu_P.$ This amounts to say that $\Gamma_{\mu_P}(u)=u$ whenever $u\in\mbox{supp }\mu_P$ satisfies $\xi''(u)>0.$ If there is some $u\in\mbox{supp }\mu_P$ such that $\xi''(u)=0,$ then the only possibility is $u=0$ and in this case, since $\partial_{x}\Phi_{\mu_P}$ is odd in $u$ from \eqref{prop4:eq1}, we have $\Gamma_{\mu_P}(0)=0.$ This completes our proof for the first assertion.

Next, let us prove the second statement. Let $u_0\in \mbox{supp }\mu_P.$ If there exists a sequence $\{u_\ell\}_{\ell\geq 1}$ of $\mbox{supp }\mu_P$ such that $\lim_{\ell\rightarrow\infty}u_\ell=u_0,$ then the first assertion, the differentiability of $\Gamma_{\mu_P}$ and the continuity of $\Gamma_{\mu_P}'$ yield $\Gamma_{\mu_P}'(u_0)=1.$ Now assume that $u_0$ is an isolated point. If $\xi''(u_0)=0,$ then we are clearly done by  \eqref{prop3:eq1}. Suppose that $\xi''(u_0)>0.$ Note that  \eqref{prop4:eq2}, \eqref{prop4:eq3} and \eqref{prop4:eq4} applied to $\mu=\mu_P$ imply $u_0<1$. 

Let $\delta\in (0,1-u_0)$ and define $a_\delta(u)=\max(\delta-|u-u_0|,0)$ on $[0,1].$ Then $a_\delta$ is a continuous function that satisfies $0\leq u+a_\delta(u)\leq 1$ for all $u\in[0,1]$ and $|a_\delta(u)-a_\delta(u')|\leq |u-u'|$ for $u,u'\in[0,1]$. Applying the mean value theorem to $u-\Gamma_{\mu_P}(u)$, we obtain
\begin{align}\label{thm0:proof:eq1}
u-\Gamma_{\mu_P}(u)&\leq \max_{u'\in[u_0,u_0+\delta]}(1-\Gamma_{\mu_P}'(u'))\delta,\,\,u\in [u_0,u_0+\delta].
\end{align}
Hence, since $u_0$ is isolated, for $\delta$ be sufficiently small we have
\begin{align*}
\int_0^1\xi''(u)(u-\Gamma_P(u))a_\delta(u)d\mu(u)&=\int_{u_0}^{u_0+\delta}\xi''(u)(u-\Gamma_P(u))a_\delta(u)d\mu(u)\\
&\leq \max_{u'\in[u_0,u_0+\delta]}(1-\Gamma_{\mu_P}'(u'))\delta\int_{u_0}^{u_0+\delta}\xi''(u)a_\delta(u)d\mu(u)\\
&=\max_{u'\in[u_0,u_0+\delta]}(1-\Gamma_{\mu_P}'(u'))\delta^2\xi''(u_0)\mu(\{u_0\}).
\end{align*}
From \eqref{thm0:proof:eq2}, this inequality implies $$
\max_{u'\in[u_0,u_0+\delta]}(1-\Gamma_{\mu_P}'(u'))\geq 0
$$ 
for sufficiently small $\delta.$ Therefore, by continuity of $\Gamma_{\mu_P}'$, we obtain $\Gamma_{\mu_P}'(u_0)\leq 1$ and this completes our proof.
\end{proof}

The rest of the section is devoted to proving Proposition \ref{prop3}. We rely on two lemmas.

 \begin{lemma}
\label{prop1} Suppose that $\mu$ is a probability measure on $[0,1]$ with continuous density $\rho(t), t \in[0,1]$. Consider $f\in C^{2,1}(\mathbb{R}\times[0,1])$ and $g\in C(\mathbb{R}\times [0,1])$ such that on $\mathbb{R}\times[0,1]$
\begin{align}
\begin{split}\label{prop1:eq1}
&\max \{ |f(x,u)|,|\partial_xf(x,u)|,|\partial_uf(x,u)|,|\partial_x^2f(x,u)|\}\leq C\exp |x|,
\end{split}\\
\begin{split}
\label{prop1:eq2}
&0\leq g(x,u)\leq C(1+|x|+u)
\end{split}
\end{align}
for some fixed constant $C>0.$ Define
\begin{align}\label{prop1:eq3}
F(u)&=\e f(M(u),u)\exp D(u)
\end{align}
for $u\in[0,1],$ where $D(u):=-\int_0^u g(M(t),t)\rho(t)dt.$ Then we have that
\begin{align}
\begin{split}\label{prop1:eq4}
F'(u)&=\e \left(\partial_uf(M(u),u)+\frac{\xi''(u)}{2}\partial_x^2f(M(u),u)\right)\exp D(u)\\
&-\rho(u)\e g(M(u),u)f(M(u),u)\exp D(u).
\end{split}
\end{align}
\end{lemma}

\begin{proof}
We will only prove that the right derivative of $F$ is equal to \eqref{prop1:eq3}. One may adapt the same argument to prove that the left derivative of $F$ is also equal to $(\ref{prop1:eq3}).$ Suppose that $0\leq u<1.$ Let $0<h<1-u.$ Write
\begin{align}
\label{prop1:proof:eq1}
F(u+h)-F(u)&=\e I_1(u) +\e I_2(u),
\end{align}
where
\begin{align*}
I_1(h)&:=(f(M(u+h),u+h)-f(M(u),u))\exp D(u),\\
I_2(h)&:=f(M(u+h),u+h)(\exp D(u+h)-\exp D(u)).
\end{align*}
It suffices to check that 
\begin{align}
\begin{split}\label{prop1:proof:eq4}
\lim_{h\downarrow 0}\frac{\e I_1(h)}{h}&=\e\left(\partial_tf(M(u),u)+\frac{\xi''(u)}{2}\partial_{x}^2f(M(u),u)\right)\exp D(u),
\end{split}\\
\begin{split}
\label{prop1:proof:eq5}
\lim_{h\downarrow 0}\frac{\e I_2(h)}{h}&=-\rho(u)\e f(M(u),u)g(M(u),u)\exp D(u).
\end{split}
\end{align}

Let us handle $(\ref{prop1:proof:eq4})$ first. Using It\^{o}'s formula, we write
\begin{align*}
f(M(u+h),u+h)-f(M(u),u)&=\int_{u}^{u+h}J(t)dt+\int_u^{u+h}\partial_xf(M(t),t)dM(t),
\end{align*}
where 
\begin{align*}
J(t):=\partial_tf(M(t),t)+\frac{1}{2}\partial_x^2 f(M(t),t)\xi''(t).
\end{align*}
Note that since $D(u)$ is independent of $(M(t)-M(u))_{u\leq t\leq u+h}$, a standard approximation argument using the left Riemann sum for $\int_u^{u+h}\partial_xf(M(t),t)dM(t)$ and  \eqref{prop1:eq1} yield that $
\e\int_u^{u+h}\partial_xf(M(t),t)dM(t)\exp D(u)=0$ and thus 
\begin{align}
\label{prop1:proof:eq3}
\frac{1}{h}\e I_1(h)&=\frac{1}{h}\e\int_u^{u+h}J(t)dt\exp D(u).
\end{align}
Define
\begin{align*}
E_0(h)&=\sup_{0<h'\leq h}\frac{1}{h'}\int_u^{u+h'}\left|J(t)\right|dt\exp D(u).
\end{align*}
Using $D\leq 0$, $(\ref{prop1:eq1})$, \eqref{prop1:eq2} and the fact that 
\begin{align}
\label{prop1:proof:eq8}
\p(\sup_{0\leq h\leq 1}|M(h)|\geq b)\leq 4\p(M(1)\geq b),\,\,b\geq 0,
\end{align}
it follows that 
\begin{align}
\begin{split}\label{prop1:proof:eq2}
\e E_0(h)&\leq C(1+\xi''(1))\e \exp \sup_{0<h'\leq h}|M(u+h')|\\
&\leq C(1+\xi''(1))\e \exp \sup_{0<h'\leq 1}|M(h')|\\
&\leq 4C(1+\xi''(1))\e \exp M(1)\\
&=4C(1+\xi''(1))\exp \frac{\xi'(1)}{2}.
\end{split}
\end{align}
Since \begin{align*}
\lim_{h\downarrow 0}\frac{1}{h}\int_u^{u+h}J(t)dt&=J(u),
\end{align*}
using $(\ref{prop1:proof:eq2})$ and the dominated convergence theorem yield 
\begin{align*}
\lim_{h\downarrow 0}\frac{1}{h}\e \int_u^{u+h}J(t)dt\exp D(u)&=\e J(u)\exp D(u).
\end{align*} 
and this combined with $(\ref{prop1:proof:eq3})$ gives $(\ref{prop1:proof:eq4}).$

\smallskip

Next we compute $(\ref{prop1:proof:eq5}).$ Define
\begin{align*}
E_1(h)&=\max_{0<h'\leq h}|f(M(u+h'),u+h')|,\\
E_2(h)&=\max_{0<h'\leq h}\frac{|\exp D(u+h')-\exp D(u)|}{h'}.
\end{align*}
From \eqref{prop1:eq1} and \eqref{prop1:proof:eq8}, 
\begin{equation}\label{prop1:proof:eq7}
\e E_1(h)^2\leq C^2\e \exp 2\max_{0<u\leq 1}|M(u)|\leq 4C^2\e \exp 2M(1)\leq 4C^2\exp 2\xi'(1).
\end{equation}
Using $D\leq 0$ and \eqref{prop1:eq2} again, the mean value theorem implies
\begin{align*}
E_2(h)&\leq \sup_{0<h'\leq h}\frac{|D(u+h)-D(u)|}{h'}\\
&\leq C\sup_{0<h'\leq h}\frac{\int_{u}^{u+h'}(1+|M(t)|+t)\rho(t)dt}{h'}\\
&\leq 2C\|\rho\|_\infty+C\|\rho\|_\infty\sup_{0\leq u\leq 1}|M(u)|
\end{align*}
and thus, the use of $(a+b)^2\leq 4a^2+4b^2$ for $a,b\in\mathbb{R}$ leads to
\begin{align}
\begin{split}\label{prop1:proof:eq6}
\e E_2(h)^2&\leq  C^2(16\|\rho\|_\infty^2+4\|\rho\|_{\infty}^2 \e \sup_{0\leq u\leq 1}|M(u)|^2)\\
&=C^2(16\|\rho\|_\infty^2+4\|\rho\|_{\infty}^2\xi'(1)).
\end{split}
\end{align}
From the Cauchy-Schwartz inequality, $(\ref{prop1:proof:eq7})$ and $(\ref{prop1:proof:eq6})$, we conclude that $\e E_1(h)E_2(h)<\infty.$ Since $\sup_{0<h'\leq h}|I_2(h')|/h'\leq E_1(h)E_2(h)$ and 
\begin{align*}
\lim_{h\downarrow 0}\frac{I_2(h)}{h}&=-\rho(u) f(M(u),u)g(M(u),u)\exp D(u),
\end{align*}
the dominated convergence theorem implies $(\ref{prop1:proof:eq5})$ and this completes our proof.
\end{proof}

In the next Lemma we use the convention that for any sequence $(a_j)_{j\geq 1}$, $\sum_{j=1}^{0} a_j = 0.$
\begin{lemma}
\label{prop2} Let $\mu\in M[0,1]$ be continuous on $[a,b]$ for some $a,b\in[0,1].$ Suppose that $L$ is  a polynomial on $\mathbb{R}^k$. Define
\begin{align}
\label{prop2:eq1}
F_\mu(u)&=\e L(\partial_x\Phi_\mu(M(u),u),\ldots,\partial_x^k\Phi_\mu(M(u),u))\exp W_\mu(u)
\end{align}
for $u\in[0,1].$ Then for $u\in [a,b],$
\begin{align}
\begin{split}
\label{prop2:eq2}
F_\mu'(u)&=\frac{\xi''(u)}{2}\e\left(\sum_{i,j=1}^k\partial_{y_i}\partial_{y_j}L(\partial_x\Phi_\mu,\ldots,\partial_x^k\Phi_\mu)\partial_x^{i+1}\Phi_\mu\partial_{x}^{j+1}\Phi_\mu\right.\\
&\left.-\mu([0,u])\sum_{i=1}^k\sum_{j=1}^{i-1}{i\choose j}\partial_{y_i}L(\partial_x\Phi_\mu,\ldots,\partial_x^k\Phi_\mu)\partial_x^{j+1}\Phi_\mu\partial_x^{i-j+1}\Phi_\mu\right)\exp W_\mu(u).
\end{split}
\end{align}
\end{lemma}

\begin{proof}
To simplify of our notation, we denote $\partial_{y_i}L$ by $L_i$ and $\partial_{y_i}\partial_{y_j}L$ by $L_{ij}.$ Also, we denote $\Phi_\mu$ by $\Phi,$  $\partial_{x}^j\Phi_\mu$ by $\Phi_{x^j}$ and $\partial_{x}^j\partial_u\Phi_\mu$ by $\Phi_{x^ju}$ provided the derivatives exist.
First we prove $(\ref{prop2:eq2})$ in the case that $\mu$ has a continuous density $\rho$ on $[0,1].$ This assumption implies that $\Phi_{x^iu}$ is continuous from $(iii)$ in Proposition \ref{thm1}. Set
\begin{align*}
f(x,u)&=L(\Phi_{x}(x,u),\ldots,\Phi_{x^n}(x,u))\exp S(x,u),\\
g(x,u)&=\Phi(x,u),
\end{align*}
where $S(x,u):=\int_0^u\rho(t)dt\Phi(x,u).$ Using Proposition \ref{thm1}, we compute that
\begin{align*}
\partial_uf&=\left(\sum_{i=1}^nL_i\Phi_{x^iu}+L\rho\Phi+L\int_0^u\rho dt\Phi_u\right)\exp S,\\
\partial_xf&=\left(\sum_{i=1}^nL_i\Phi_{x^{i+1}}+L\int_0^u\rho dt \Phi_x\right)\exp S,\\
\partial_x^2f&=\left(\sum_{i,j=1}^n L_{ij}\Phi_{x^{i+1}}\Phi_{x^{j+1}}+\sum_{i=1}^n L_i\Phi_{x^{i+2}}+L\int_0^u\rho dt\Phi_{x^2}\right)\exp S\\
&+\int_0^u\rho dt\left(2\sum_{i=1}^n L_i\Phi_{x^{i+1}}\Phi_x+L\int_0^u\rho dt(\Phi_x)^2\right)\exp S.
\end{align*}
Recall that $\Phi$ satisfies the Parisi PDE
\begin{align*}
\Phi_u&=-\frac{\xi''}{2}\left(\Phi_{x^2}+\int_0^u\rho dt (\Phi_x)^2\right).
\end{align*}
Taking $i$-th partial derivative with respect to the $x$ variable yields
\begin{align*}
\Phi_{x^iu}&=-\frac{\xi''}{2}\left(\Phi_{x^{i+2}}+\int_0^u\rho dt \sum_{j=0}^i{i\choose j}\Phi_{x^{j+1}}\Phi_{x^{i-j+1}}\right).
\end{align*}
Therefore, we have that
\begin{align*}
&\partial_uf+\frac{\xi''}{2}\partial_{x}^2f-\rho gf\\
&=L\int_0^u\rho dt\left(\Phi_u+\frac{\xi''}{2}\left(\Phi_{x^2}+\int_0^u\rho dt(\Phi_x)^2\right)\right)\exp S+\sum_{i=1}^n L_i\Phi_{x^iu}\exp S\\
&+\frac{\xi''}{2}\left(\sum_{i,j=1}^n L_{ij}\Phi_{x^{i+1}}\Phi_{x^{j+1}}+\sum_{i=1}^nL_i\Phi_{x^{i+2}}+2\int_0^u\rho dt\sum_{i=1}^nL_i\Phi_{x^{i+1}}\Phi_x\right)\exp S\\
&=-\frac{\xi''}{2}\sum_{i=1}^nL_i\left(\Phi_{x^{i+2}}+\int_0^u\rho dt\sum_{j=0}^i{i\choose j}\Phi_{x^{j+1}}\Phi_{x^{i-j+1}}\right)\exp S\\
&+\frac{\xi''}{2}\left(\sum_{i,j=1}^n L_{ij}\Phi_{x^{i+1}}\Phi_{x^{j+1}}+\sum_{i=1}^nL_i\Phi_{x^{i+2}}+2\int_0^u\rho dt\sum_{i=1}^nL_i\Phi_{x^{i+1}}\Phi_x\right)\exp S\\
&=\frac{\xi''}{2}\left(\sum_{i,j=1}^nL_{ij}\Phi_{x^{i+1}}\Phi_{x^{j+1}}-\int_0^u\rho dt\sum_{i=1}^n\sum_{j=1}^{i-1}{i\choose j}L_i\Phi_{x^{j+1}}\Phi_{x^{i-j+1}}\right)\exp S\ .
\end{align*}
Applying Lemma $\ref{prop1}$, our assertion clearly follows in this case that $\mu$ has continuous density on $[a,b].$ Next, we assume that $\mu$ is continuous on $[0,1].$ Pick a sequence of probability measures $(\mu_n)_{n\geq 1}$ on $[0,1]$ with continuous densities that converges to $\mu$ weakly. Using the continuity of $\mu$ on $[a,b]$, we can further assume that $\lim_{n\rightarrow\infty}\sup_{a\leq u\leq b}|\mu_n([0,u])-\mu([0,u])|=0.$ Let $F_{\mu_1},F_{\mu_2},\ldots, F_\mu$ be defined as $(\ref{prop2:eq1})$ by using $\mu_1,\mu_2,\ldots,\mu$, respectively. Using the weak convergence of $(\mu_n)_{n\geq 1}$ and Proposition $\ref{thm1}$, we know that $(F_{\mu_n})_{n\geq 1}$ converges to $F_\mu$ uniformly on $[0,1].$ On the other hand, by our special choice of $(\mu_n)_{n\geq 1}$ and Theorem $\ref{thm1}$, $(F_{\mu_n}')_{n\geq 1}$ converges uniformly on $[a,b].$ These facts imply that on $[a,b],$ $F_\mu$ is differentiable and $F_\mu'$ is given by $(\ref{prop2:eq2}).$ This completes our proof.
\end{proof}

\begin{proof}[Proof of Proposition \ref{prop3}]
Let us pick a sequence of probability measures $(\mu_n)_{n\geq 1}$ with continuous densities that satisfies $\lim_{n\rightarrow\infty}\mu_n([0,u])=\mu([0,u])$ for all $0\leq u\leq 1.$ An application of Lemma \ref{prop2} with $k=1$ and $L(y_1)=y_1^2$ yields that
\begin{align*}
\Gamma_{\mu_n}'(u)&=\xi''(u)\e (\partial_x^2\Phi_{\mu_n}(M(u),u))^2\exp W_{\mu_n}(u).
\end{align*}
Another application of Lemma \ref{prop2} with $k=2$ and $L(y_1,y_2)=y_2^2$ implies
\begin{align*}
\Gamma_{\mu_n}''(u)&=\gamma_{1,\mu_n}(u)-\mu_n([0,u])\gamma_{2,\mu_n}(u),
\end{align*}
where 
\begin{align*}
\gamma_{1,{\mu_n}}(u)&=\xi'''(u)\e (\partial_x^2\Phi_{\mu_n}(M(u),u))^2\exp W_{\mu_n}(u)\\
&+\xi''(u)^2\e (\partial_{x}^3\Phi_{\mu_n}(M(u),u))^2\exp W_{\mu_n}(u)
\end{align*}
and
\begin{align*}
\gamma_{2,{\mu_n}}(u)&=2\xi''(u)^2\e(\partial_x^2\Phi_{\mu_n}(M(u),u))^3\exp W_{\mu_n}(u).
\end{align*}
Since $(\Gamma_{\mu_n}')_{n\geq 1}$ converges uniformly to $\xi''(\cdot)\e(\partial_x^2\Phi_\mu(M(\cdot),\cdot))^2\exp W_\mu(\cdot)$ on $[0,1]$, it implies that $\Gamma_\mu=\lim_{n\rightarrow\infty}\Gamma_{\mu_n}$ is differentiable and its derivative is given by $(\ref{prop3:eq1}).$ Now, let $0\leq u_1<u_2\leq 1.$ Suppose that $u_1'$ and $u_2'$ satisfy $u_1<u_1'<u_2'<u_2.$ From the mean value theorem, we can write
\begin{align}\label{prop3:proof:eq1}
\frac{\Gamma_{\mu_n}'(u_2')-\Gamma_{\mu_n}'(u_1')}{u_2'-u_1'}&=\Gamma_{\mu_n}''(u_0),
\end{align}
for some $u_0\in (u_1',u_2')$. Note that
\begin{itemize}
\item $\mu_n([0,u_1'])\leq \mu_n([0,u])\leq \mu_n([0,u_2'])$ for $u\in [u_1',u_2'].$
\item $\lim_{n\rightarrow\infty}\mu_n([0,u_1'])=\mu([0,u_1'])$ and $\lim_{n\rightarrow\infty}\mu_n([0,u_2'])=\mu([0,u_2']).$
\item $\gamma_{1,\mu}=\lim_{n\rightarrow\infty}\gamma_{1,{\mu_n}}$ and $\gamma_{2,\mu}=\lim_{n\rightarrow\infty}\gamma_{2,{\mu_n}}$ uniformly, by part $(ii)$ of Propostion 1.
\end{itemize}
They together with $(\ref{prop3:proof:eq1})$ imply
\begin{align*}
\frac{\Gamma_\mu'(u_2')-\Gamma_\mu'(u_1')}{u_2'-u_1'}&\leq\max_{u\in [u_1',u_2']}\gamma_{1,\mu}(u)-\mu([0,u_1'])\min_{u\in[u_1',u_2']}\gamma_{2,\mu}(u)
\end{align*}
and
\begin{align*}
\frac{\Gamma_\mu'(u_2')-\Gamma_\mu'(u_1')}{u_2'-u_1'}&\geq\min_{u\in [u_1',u_2']}\gamma_{1,\mu}(u)-\mu([0,u_2'])\max_{u\in [u_1',u_2']}\gamma_{2,\mu}(u).
\end{align*}
Now letting $u_1'\downarrow u_1$ and $u_2'\uparrow u_2$, we obtain 
\begin{align*}
\frac{\Gamma_\mu'(u_2)-\Gamma_\mu'(u_1)}{u_2-u_1}&\leq\max_{u\in [u_1,u_2]}\gamma_{1,\mu}(u)-\mu([0,u_1])\min_{u\in[u_1,u_2]}\gamma_{2,\mu}(u)
\end{align*}
and
\begin{align*}
\frac{\Gamma_\mu'(u_2)-\Gamma_\mu'(u_1)}{u_2-u_1}&\geq\min_{u\in [u_1,u_2]}\gamma_{1,\mu}(u)-\mu([0,u_2))\max_{u\in [u_1,u_2]}\gamma_{2,\mu}(u).
\end{align*}
Since the distribution of $\mu$ is right continuous and $\gamma_{1,\mu},\gamma_{2,\mu}$ are continuous, $(\ref{prop3:eq2})$ follows by applying $u_1=u$ and $u_2=u+h$ with $h\downarrow 0$ to these two inequalities. Also, letting $u_1=u+h$ with $h\uparrow 0$ and $u_2=u$ gives $(\ref{prop3:eq3})$.
\end{proof}

% % % % % % % % % % % % % % % % % % % % % % % % % % % % % % % % % % % % % % % % % % % % % % % % % % % % % % % % % % % % % % % % % % % % % % % % % % % % % %% % % % % % % % % % % % % % % % % % % % % % % % % % % % % % % % % % % % % % % % % % % % % % % % % % % % % % % % % % % % % % % % % % % % % % % % % % % % % %% % % % % % % % % % % % % % % % % % % % % % % % % % % % % % % % % % % % % % % % % % % % % % % % % % % % % % % % % % % % % % % % % % % % % % % % % % % % % %% % % % % % % % % % % % % % % % % % % % % % % % % % % % % % % % % % % % % % % % % % % % % % % % % % % % % % % % % % % % % % % % % % % % % % % % % % 

\section{Proofs of Theorems \bf\ref{thm:theorem1}, \ref{thm:continuity}, \ref{sec3:thm1} and \ref{thm:top}}\label{Section:s} 

In this section we will prove our main theorems stated in Section \ref{intro}.

\subsection{Proof of Theorem  \ref{thm:theorem1}}

\begin{proof}[Proof of Theorem \ref{thm:theorem1}]
We start by proving item $(i)$. Suppose that $q_m:=\min\mbox{supp }\mu_P\neq 0.$ Note that $\Gamma_{\mu_P}(q_m)=q_m$ from Theorem \ref{thm0}. Also since $\partial_x\Phi_{\mu_P}(x,u)$ is an odd function in $u$ by \eqref{prop4:eq1}, this implies that $\partial_x\Phi_{\mu_P}(x,0)=0$ and then $\Gamma_{\mu_P}(0)=0.$ Now since $\mu([0,u])=0$ for $0\le u<q_m,$ Proposition $\ref{prop3}$ implies the differentiability of $\Gamma_{\mu_P}'$ on $[0,q_m)$ and moreover with the help of \eqref{prop4:eq3},
\begin{align*}
\Gamma_{\mu_P}''(u)&=\gamma_{1,\mu_P}(u)>0
\end{align*}
for $0<u<q_m.$ This means that from \eqref{prop4:eq3} and \eqref{prop4:eq4}, $\Gamma_{\mu_P}'(u)<\Gamma_{\mu_P}'(q_m)\leq 1$ on $[0,q_m).$ So $\Gamma_{\mu_P}$ can have only one fixed point on $[0,q_m],$ which contradicts $\Gamma_{\mu_P}(0)=0, \Gamma_{\mu_P}(q_m)=q_m.$ This gives $(i).$ 
\smallskip

Next let us turn to the proof of $(ii).$ Suppose that $\mu_P((0,q])>0$ for some $0<q<\hat{q}.$ Let us take $q'\in\mbox{supp }\mu_P\cap (0,q].$ Then  $\Gamma_{\mu_P}(q')={q'}$ from Theorem \ref{thm0}. Note that from the discussion above, we also have $\Gamma_{\mu_P}(0)=0.$ Using the mean value theorem to $\Gamma_{\mu_P}$ and \eqref{prop4:eq3}, we obtain a contradiction,
\begin{align*}
1=\frac{\Gamma_{\mu_P}(q')-\Gamma_{\mu_P}(0)}{q'-0}=\Gamma_{\mu_P}'(q'')\leq \xi''(q'')<\xi''(\hat{q})=1
\end{align*}
for some $q''\in(0,q').$ Hence, $\mu_P((0,q])=0$ for all $0<q<\hat{q}$ and this together with $(i)$ gives $(ii).$
\end{proof}

\subsection{Proof of Theorem \bf\ref{thm:continuity}}

\begin{proof}[Proof of Theorem \ref{thm:continuity}]
We prove $(i)$ first. Since $(u_\ell^+)_{\ell\geq 1},(u_\ell^-)_{\ell\geq 1}\subseteq \mbox{supp }\mu_P,$ we have by Theorem \ref{thm0}, $\Gamma_{\mu_P}(u_\ell^+)=u_\ell^+$ and $\Gamma_{\mu_P}(u_\ell^-)=u_\ell^-$. The mean value theorem and $(i)$ in Proposition \ref{prop3} now ensures the existence of two sequences $(\hat{u}_\ell^+)_{\ell\geq 1}$ and $(\hat{u}_\ell^-)_{\ell\geq 1}$ that satisfy $\hat{u}_{\ell}^+\downarrow u_0$, $\hat{u}_\ell^-\uparrow u_0$ and $\Gamma_{\mu_P}'(\hat{u}_\ell^+)=1=\Gamma_{\mu_P}'(\hat{u}_\ell^-).$ These together with $(\ref{prop3:eq2})$ and $(\ref{prop3:eq3})$ imply that 
\begin{align*}
\gamma_{1,\mu_P}(u_0)-\mu_P([0,u_0))\gamma_{2,\mu_P}(u_0)&=\lim_{h\rightarrow 0+}\frac{\Gamma_{\mu_P}'(u_0+h)-\Gamma_{\mu_P}'(u_0)}{h}=0\\
\gamma_{1,\mu_P}(u_0)-\mu_P([0,u_0])\gamma_{2,\mu_P}(u_0)&=\lim_{h\rightarrow 0-}\frac{\Gamma_{\mu_P}'(u_0+h)-\Gamma_{\mu_P}'(u_0)}{h}=0,
\end{align*} 
where $\gamma_{1,\mu_P}$ and $\gamma_{2,\mu_P}$ are defined as in Proposition $\ref{prop3}$.
Since $\gamma_{2,\mu_P}(u_0)\neq 0$ from \eqref{prop4:eq3}, it follows that $\mu_P([0,u_0])=\mu_P([0,u_0))$ and so $\mu_P$ is continuous at $u_0.$

As for $(ii),$ we denote by $x_{\mu_P}$ the distribution of $\mu_P.$ Note that since $(a,b)\subseteq \mbox{supp }\mu_P$, $(i)$ implies the continuity of $x_{\mu_P}$ on $(a,b)$ and thus the right continuity of $x_{\mu_P}$ further gives the continuity of $x_{\mu_P}$ on $[a,b).$ We claim that $x_{\mu_P}$ is infinitely differentiable on $[a,b)$ by induction. Since $(a,b)\subset\mbox{supp }\mu_P$, Theorem \ref{thm0} and continuity of $\Gamma_{\mu_P}$ yield $\Gamma_{\mu_P}(u)=u$ on $[a,b)$. Therefore, continuity of $\mu_P$  Proposition \ref{prop3} implies $\Gamma_{\mu_P}''(u)=0$ on $[a,b).$ Consequently, it  gives us $x_{\mu_P}(u)\gamma_{2,\mu_P}(u)=\gamma_{1,\mu_P}(u)$ on $[a,b).$ Note again that $\gamma_{2,\mu_P}(u)\neq 0$ on $[0,1].$ We may now write
\begin{align}\label{thm3:proof:eq1}
x_{\mu_P}(u)&=\frac{\zeta(u)F_1(u)+F_2(u)}{F_3(u)} ,
\end{align}
where 
\begin{align*}
\zeta(u)&:=\xi'''(u)/\xi''(u)^2\ ,\\
F_1(u)&:=\e\partial_x^2\Phi_{\mu_P}(M(u),u)^2\exp W_{\mu_P}(u) ,\\
F_2(u)&:=\e (\partial_x^3\Phi_{\mu_P}(M(u),u)^2\exp W_{\mu_P}(u) ,\\
F_3(u)&:=2\e (\partial_x^2\Phi_{\mu_P}(M(u),u))^2\exp W_{\mu_P}(u) .
\end{align*}
Now since $\mu_P$ is continuous on $[a,b),$ Lemma $\ref{prop2}$ implies that $F_1,F_2,F_3$ are differentiable on $[a,b)$. We then conclude that $x_{\mu_P}$ is differentiable on $[a,b).$ Suppose that $x_{\mu_P}^{(n)}$ exists on $[a,b).$ Observe that from $(\ref{thm3:proof:eq1}),$ one can easily derive by differentiating $F_i$ for $j\leq n$ times,
\begin{align*}
F_i^{(j)}(u)&=\e L_{i,j}(\xi'',\ldots,\xi^{(j+2)},\partial_x\Phi_{\mu_P},\ldots,\partial_x^{j+3}\Phi_{\mu_P},x_{\mu_P},x_{\mu_P}',\ldots,x_{\mu_P}^{(j-1)}),
\end{align*}
where $L_{i,j}$'s are polynomials of $3j+4$ variables. Applying $(\ref{prop2:eq2})$ again and using the induction hypothesis that $x_{\mu_P}^{(n)}$ exists, it follows that $F_i^{(n+1)}$ exists and the quotient rule completes our proof.
\end{proof}

\subsection{Proof of Theorem  \ref{sec3:thm1}}

We will prove Theorem \ref{sec3:thm1} by contradiction. Before we turn to the main proof, let us make a few observations on Parisi measures. Denote by $Z_{N,t}$ the partition function associated to the Hamiltonian $tH_N$ for $t\geq 0$, that is, 
$$
Z_{N,t}=\sum_{\vsi}\exp tH_N(\vsi).
$$ 
Denote by $\left<\cdot\right>$ the expectation with respect to the Gibbs measure $G_N$ corresponding to the Hamiltonian $H_N$ in Section \ref{intro}. A direct differentiation yields
\begin{align}\label{sec3:eq1}
\left.\frac{d}{dt}\frac{1}{N}\e\log Z_{N,t}\right|_{t=1}&=\frac{1}{N}\e\left<H_N(\vsi)\right>\leq \frac{1}{N}\e \max_{\vsi}H_N(\vsi)\leq \sqrt{2\xi(1)\log 2}.
\end{align}
Here the last inequality in $(\ref{sec3:eq1})$ relies on a standard Gaussian inequality that $\e \max_{i\leq M}g_i\leq \tau\sqrt{2\log M}$ for arbitrary centered Gaussian process $(g_i)_{i\leq M}$ with $\e g_i^2\leq \tau^2$ for $i\leq M.$ Now, the Gaussian integration by parts applied to $\e\left<H_N(\vsi)\right>$ implies that 
\begin{align*}
\frac{1}{N}\e\left<H_N(\vsi)\right>&=\e\left<\xi(1)-\xi(R(\vsi^1,\vsi^2))\right>
\end{align*}
and from $(\ref{sec3:eq1})$,
\begin{align}\label{sec3:eq2}
\e\left<\frac{\xi(R(\vsi^1,\vsi^2))}{\xi(1)}\right>\geq 1-\sqrt{\frac{2\log 2}{\xi(1)}}.
\end{align}
It is well-known \cite{Pan11:2,Tal07} that the moments of a Parisi measure contain information of the limit of the overlap under $\e\left<\cdot\right>$ through 
\begin{align*}
\lim_{N\rightarrow\infty}\e\left<R(\vsi^1,\vsi^2)^p\right>=\int_0^1q^pd\mu_P(q)
\end{align*}
for all $p\geq 2$ with $\beta_p\neq 0.$ This and \eqref{sec3:eq2} imply that
\begin{align}\label{thm3:eq1}
\int_0^1\frac{\xi(q)}{\xi(1)}d\mu_P=\lim_{N\rightarrow\infty}\e\left<\frac{\xi(R(\vsi^1,\vsi^2))}{\xi(1)}\right>\geq 1-\sqrt{\frac{2\log 2}{\xi(1)}}.
\end{align}

Now suppose on the contrary that $\mu_P$ is either RS or 1RSB. If $\mu_P$ is RS, part $(i)$ in Theorem \ref{thm:theorem1} implies that $\mu_P=\delta_0$. However, this contradicts \eqref{thm3:eq1} since the left side of \eqref{thm3:eq1} is equal to zero, while the right side of the same equation is positive by \eqref{sec3:thm1:eq1}. Now suppose that $\mu_P$ is 1RSB, that is, $\mu_P$ consists of exactly two atoms. Again by part $(i)$ of Theorem \ref{thm:theorem1}, we may assume that $\mu_P=\hat{m}\delta_0+(1-\hat{m})\delta_{\hat{q}}$ for some $0<\hat{m},\hat{q}<1$. Plugging $\mu_P$ into \eqref{thm3:eq1} gives
\begin{align*}
\frac{\xi(\hat{q})}{\xi(1)}(1-\hat{m})\geq 1-\sqrt{\frac{2\log 2}{\xi(1)}}.
\end{align*}
Observe that the left side of this inequality is bounded above by $1-\hat{m}$ and since $u\xi'(u)\geq \xi(u)$ for all $u\geq 0,$ it is also bounded above by $\xi'(\hat{q})/\xi(1).$ We conclude that $\hat{m}$ and $\hat{q}$ must satisfy the following two inequalities,
\begin{align}
\begin{split}
\label{sec1:eq3}
\hat{m}&\leq \sqrt{\frac{2\log 2}{\xi(1)}}
\end{split}
\end{align}
and
\begin{align}
\begin{split}
\label{sec1:eq2}
\xi'(\hat{q})&\geq \xi(1)\left(1-\sqrt{\frac{2\log 2}{\xi({1})}}\right)=\xi(1)-\sqrt{2\xi(1)\log 2}.
\end{split}
\end{align}

\begin{proof} [Proof of Theorem \ref{sec3:thm1}] Note that $\mu_P$ corresponds to $\mathbf{m}=(0,\hat{m},1)$ and $\mathbf{q}=(0,0,\hat{q},1)$ as described in Section \ref{section:pde}. Let $z_0,z_1,z_2$ be independent centered Gaussian random variables with $\e z_0^2=0,$ $\e z_1^2=\xi'(\hat{q})$ and $\e z_2^2=\xi'(1)-\xi'(\hat{q}).$ Then using \eqref{eq:ParisiRecursion},
\begin{align*}
X_3&=\log \cosh (z_0+z_1+z_2)=\log\cosh(z_1+z_2),\\
X_2&=\log\cosh z_1+\frac{1}{2}(\xi'(1)-\xi'(\hat{q})),\\
X_1&=X_0=\frac{1}{\hat{m}}\log\e \exp \hat{m} X_2=\frac{1}{\hat{m}}\log \e \cosh^{\hat{m}}z_1+\frac{1}{2}(\xi'(1)-\xi'(\hat{q})).
\end{align*}
Plugging $X_1$ and $X_2$ into the definition \eqref{sec3:gamma} of $\Gamma_{\mu_P}$ and using Proposition \ref{prop3} we obtain 
\begin{align}\label{eq:banana}
\Gamma_{\mu_P}'(\hat{q})&=\xi''(\hat{q})\e\frac{\exp \hat{m}(X_2-X_1)}{\cosh^4z_1}=\xi''(\hat{q})\frac{\e\cosh^{\hat{m}-4}z_1}{\e\cosh^{\hat{m}}z_1}.
\end{align}
Let us recall two useful facts about a Gaussian random variable,
\begin{align}
\begin{split}
\label{sec3:eq4}
\e e^{a|g|}&=2e^{\frac{a^2}{2}}\phi(a),\,\,\forall a\in\mathbb{R},
\end{split}
\end{align}
and
\begin{align}
\begin{split}\label{sec3:eq7}
\frac{3}{4|a|}&\leq e^{\frac{a^2}{2}}\phi(a)\leq \frac{1}{|a|},\,\,\forall a\leq -2, 
\end{split}
\end{align}
where $\phi(a)=\int_{-a}^\infty e^{-\frac{x^2}{2}}/\sqrt{2\pi}dx$ for $a\in\mathbb{R}.$ Note that since $0\leq \hat{m}\leq 1$ and $\xi(1)\geq 8\log 2,$ it follows from $(\ref{sec1:eq2})$ that $({4-\hat{m}})\sqrt{\xi'(\hat{q})}\geq 2.$ Also note that $\cosh x\leq e^{|x|}.$ Now using $(\ref{sec3:eq4})$ and $(\ref{sec3:eq7})$ with $a=-(4-\hat{m})\sqrt{\xi'(\hat{q})}$ we obtain
\begin{align*}
\e\cosh^{\hat{m}-4}z_1&\geq \e \exp (\hat{m}-4)|z_1|\\
&=2\Phi\left((\hat{m}-4)\sqrt{\xi'(\hat{q})}\right)\exp \frac{1}{2}(\hat{m}-4)^2\xi'(\hat{q})\\
&\geq \frac{3}{2}\frac{1}{({4-\hat{m}})\sqrt{\xi'(\hat{q})}}\\
&\geq \frac{3}{8\sqrt{\xi'(1)}}.
\end{align*}
On the other hand, $(\ref{sec3:eq4})$ with $a=\hat{m}\sqrt{\xi'(\hat{q})}$ gives
\begin{align*}
\e\cosh^{\hat{m}}z_1&\leq \e \exp \hat{m}|z_1|\\
&=2\Phi\left(\hat{m}\sqrt{\xi'(\hat{q})}\right)\exp \frac{\hat{m}^2}{2}\xi'(\hat{q})\\
&\leq 2\exp \frac{\hat{m}^2}{2}\xi'(\hat{q})\\
&\leq 2\exp \frac{\hat{m}^2}{2}\xi'(1).
\end{align*} 
From these two inequalities and \eqref{eq:banana}, we have
\begin{align}\label{sec3:eq6}
\frac{3}{16}\frac{\xi''(\hat{q})}{\sqrt{\xi'({1})}\exp \frac{\hat{m}^2}{2}\xi'(1)}\leq \xi''(\hat{q})\frac{\e\cosh^{\hat{m}-4}z_1}{\e\cosh^{\hat{m}}z_1}=\Gamma_{\mu_P}'(\hat{q}).
\end{align}
Next, note that $u\xi''(u)\geq \xi'(u)$ and $\sqrt{\xi(1)}-\sqrt{2\log 2}\geq  \sqrt{\xi(1)}/2$ since we assumed in \eqref{sec3:thm1:eq1} that $\xi(1)\geq 8\log 2$. From $(\ref{sec1:eq2})$, 
\begin{align}\label{sec3:eq5}
\xi''(\hat{q})=\frac{\hat{q}\xi''(\hat{q})}{\hat{q}}\geq \frac{\xi'(\hat{q})}{\hat{q}}\geq \xi'(\hat{q})\geq \sqrt{\xi(1)}(\sqrt{\xi(1)}-\sqrt{2\log 2})\geq \frac{\xi(1)}{2}.
\end{align}
From $(\ref{sec1:eq3}),$ 
\begin{align}\label{sec3:eq3}
\hat{m}^2\xi'(1)\leq \frac{2\xi'(1)\log 2}{\xi(1)}.
\end{align}
Combining $(\ref{sec3:eq5})$, $(\ref{sec3:eq3})$ and using Theorem \ref{thm0},  we conclude from $(\ref{sec3:eq6})$ that
\begin{align*}
\frac{3}{32}\frac{\xi(1)}{\sqrt{\xi'(1)}2^{\frac{\xi'(1)}{\xi(1)}}}=\frac{3}{16}\frac{\frac{\xi(1)}{2}}{\sqrt{\xi'(1)}\exp \left(\frac{\xi'(1)}{\xi(1)}\log 2\right)}\leq \Gamma_{\mu_P}'(\hat{q})\leq 1.
\end{align*}
However, this contradicts the assumption \eqref{sec3:thm1:eq1} on $\xi.$
\end{proof}

\subsection{Proof of Theorem  \ref{thm:top}}

\begin{proof}[Proof of Theorem \ref{thm:top}] If $q_{M}$ is an isolated point of $\mbox{supp }\mu_P$, it must be a jump discontinuity of $\mu_P$ and this clearly implies our assertion. Assume that $q_{M}$ is not isolated and $\mu_P$ is continuous at this point. Theorem \ref{thm0}, the mean value theorem to $\Gamma_{\mu_P}$ and continuity of $\Gamma_{\mu_P}$ imply
\begin{align}\label{thm:top:proof:eq2}
\xi''(q_{M})\e (\partial_x^2\Phi_{\mu_P}(M(q_{M}),q_{M}))^2\exp W_{\mu_P}(q_{M})=\Gamma_{\mu_P}'(q_{M})=1
\end{align}
and in addition from \eqref{prop3:eq2} and \eqref{prop3:eq3}, 
\begin{align}
\label{thm:top:proof:eq1}
\gamma_{1,\mu_P}(q_{M})-\gamma_{2,\mu_P}(q_{M})=\gamma_{1,\mu_P}(q_{M})-\mu_P([0,q_{M}])\gamma_{2,\mu_P}(q_{M})=\Gamma_{\mu_P}''(q_{M})=0.
\end{align}
Observe that $\Phi_{\mu_P}(x,q_{M})=\log \cosh x+(\xi'(1)-\xi'(q_{M}))/2.$ A straightforward computation yields
\begin{align*}
\partial_x^2\Phi_{\mu_P}(x,q_{M})&=\frac{1}{\cosh^2 x},\\
(\partial_x^3\Phi_{\mu_P}(x,q_{M}))^2&=\frac{4}{\cosh^4x}-\frac{4}{\cosh^6x}.
\end{align*}
Thus, we obtain from \eqref{thm:top:proof:eq2},
\begin{align}
\label{thm:top:proof:eq3}
\e\frac{\exp W_{\mu_P}(q_{M})}{\cosh^4M(q_{M})}&=\frac{1}{\xi''(q_{M})}.
\end{align}
Also since
\begin{align*}
\gamma_{1,\mu_P}(q_{M})&=\xi'''(q_{M})\e\frac{\exp W_{\mu_P}(q_{M})}{\cosh^4 M(q_{M})}\\
&+4\xi''(q_{M})^2\left(\e \frac{\exp W_{\mu_P}(q_{M})}{\cosh^4M(q_{M})}-\e\frac{\exp W_{\mu_P}(q_{M})}{\cosh^6M(q_{M})}\right)
\end{align*}
and
\begin{align*}
\gamma_{2,\mu_P}(q_{M})&=2\xi''(q_{M})^2\e\frac{\exp W_{\mu_P}(q_{M})}{\cosh^6M(q_{M})},
\end{align*}
they imply from \eqref{thm:top:proof:eq1} that 
\begin{align}\label{thm:top:proof:eq4}
\e \frac{\exp W_{\mu_P}(q_{M})}{\cosh^6M(q_{M})}&=\left(\frac{\xi'''(q_{M})}{6\xi''(q_{M})^2}+\frac{2}{3}\right)\e\frac{\exp W_{\mu_P}(q_{M})}{\cosh^4 M(q_{M})}.
\end{align}
Note that $\e \exp W_{\mu_P}(q_{M})=1$ from \eqref{prop4:eq4}. Using Jensen's inequality together with \eqref{thm:top:proof:eq3} and \eqref{thm:top:proof:eq4} gives
\begin{align*}
\frac{1}{\xi''(q_{M})}&=\e\frac{\exp W_{\mu_P}(q_{M})}{\cosh^4M(q_{M})}\leq \left(\e\frac{\exp W_{\mu_P}(q_{M})}{\cosh^6M(q_{M})}\right)^{2/3}
=\frac{1}{\xi''(q_{M})^{2/3}}\left(\frac{\xi'''(q_{M})}{6\xi''(q_{M})^2}+\frac{2}{3}\right)^{2/3}.
\end{align*}
One may simplify this inequality to get equivalently
\begin{align}\label{thm:top:proof:eq5}
1\leq \frac{\xi'''(q_{M})}{6(\xi''(q_{M}))^{3/2}}+\frac{2}{3}\sqrt{\xi''(q_{M})}.
\end{align}
Now since $\xi''(1)\geq \xi''(q_{M})\geq 2\beta_2^2>1$ and $\xi'''(1)\geq \xi'''(q_{M})$, \eqref{thm:top:proof:eq5} yields
\begin{align*}
1<\frac{\xi'''(1)}{6}+\frac{2}{3}\sqrt{\xi''(1)}
\end{align*}
which contradicts the assumption \eqref{thm:top:eq1}. This finishes our proof.
\end{proof}

\section{The spherical case}\label{section:spherical}

We now discuss the analogue of our results to the spherical mixed $p$-spin model. In this section, we set the configuration space to be $$\Sigma_N^s =  \bigg \{\vsi \in \R^N \bigg | \sum_{i=1}^{N} \sigma_i^2 = N \bigg \} .$$ On the sphere $\Sigma_N^s$ we consider the same Hamiltonian $H_{N}$ as in  \eqref{eq:bqs}. The spherical mixed $p$-spin was introduced by Crisanti-Sommers \cite{CriSom} as a possible simplification of the mixed $p$-spin model in the hypercube $\Sigma_N$. The main difference from the model with Ising spin configurations is that the analogous Parisi functional has a much simpler formula. This formula was discovered by Crisanti-Sommers \cite{CriSom} and proved by Talagrand \cite{Tal07s} and Chen \cite{Chen12}.  We describe it now. As before, given a probability measure $\mu$ on $[0,1]$, consider its distribution function $x_{\mu}(q) = \mu([0,q]).$
For $q \in [0,1]$, let

\begin{equation*}
\hat{x}_\mu(q) = \int_{q}^{1} x_\mu(s) d s .
\end{equation*}
Assuming that $x_\mu(\hat{q}) = 1$ for some $\hat{q} < 1$, define
\begin{equation*}\label{eq:ParisifunctionalSphere}
\mathcal P^s(\mu) = \frac{1}{2}\bigg(\int_{0}^{1} x_\mu(q) \xi'(q) d q  + \log(1-\hat{q})\bigg)  .
\end{equation*}
Otherwise, set $\mathcal P^s(\mu) = \infty.$

A measure that minimizes $\mathcal P^s$ is called a Parisi measure for the spherical mixed $p$-spin model. The above formula provides two major simplifications compared to \eqref{eq:ParFunctionalFormula}. First, it is known that, for all choices of $\xi$, Parisi measures are unique \cite[Theorem 1.2]{Tal07s}. Second, in the pure $p$-spin model ($\xi(x)=\beta_p^2 x^p$), there exists a $\beta_{p,c} >0$  such that the Parisi measure is RS below $\beta_{p,c}$ and 1RSB for all values of $\beta_p>\beta_{p,c}$ \cite[Proposition 2.2]{Tal07s}. However, for the mixed $p$-spin model, the structure of the Parisi measure is still not known and it is expected \cite{CriLeuzzi06} that the model is FRSB for a certain class of mixtures $\xi$.

We now describe our results for the spherical mixed $p$-spin model. Recall that $\xi(u) = \sum_{p\geq 2}\beta_p^2 u^p$ and assume that $x_\mu(\hat{q}) = 1$ for some $\hat{q} < 1$. Define for $0\leq q\leq\hat{q}$, 
\begin{equation}\label{eq:definitionfandF}
F(q) =\xi'(q) - \int_0^q\frac{ ds}{\hat{x}_\mu(s)^2} , \quad  f(q)=\int_{0}^q F(s)d s,
\end{equation}
and let 
\begin{equation*}S_{\hat{q}}:=  \{s\in[0,\hat{q}]|f(s) = \max_{t \in [0,\hat{q}]} f(t) \}.
\end{equation*}

Note that $S_{\hat{q}}$ depends on the distribution function $x_\mu$. It is known however that there exists $q_1<1$ depending only on $\xi$ such that $\text{supp } \mu_P \subseteq [0,q_1] $  (see discussion on page $6$ of \cite{Tal07s}). We will denote this $q_1$ by $q_1 (\xi)$ and define $S=S_{q_1(\xi)}$. The following characterization of the Parisi measure was proved in Talagrand \cite[Proposition 2.1]{Tal07s}. It mainly relies on the Crisanti-Sommers formula.

\begin{proposition}\label{prop:characterizationSphere}  $\mu_P$  is a Parisi measure if and only if $\mu_P(S)=1$.
\end{proposition}

Using this proposition, we have the following result.

\begin{theorem}\label{th:spherical} Let $\mu_P$ be a Parisi measure. Then the following hold.
\begin{enumerate}
\item[$(i)$] $\mbox{supp }\mu_P \subseteq S.$
\item[$(ii)$] If $(a,b) \subset  \mbox{supp }\mu_P$ with $0\leq a < b \leq 1$, then  $$\mu_P([0,u]) = \frac{\xi'''(u)}{2\xi''(u)^{\frac{3}{2}}}$$ for every $u \in (a,b)$.   Therefore, the distribution of $\mu_P$ is $C^\infty$ on $(a,b)$.

\item[$(iii)$] If $\beta_2 \neq 1/\sqrt{2}$ and $0\in \mbox{supp }\mu_P$, then there exists $ \hat{q} >0$ such that $\mu_P([0,\hat{q}])=\mu_P(\{0\})$.

\item[$(iv)$] Suppose that there exist an increasing sequence $(u_\ell^-)_{\ell\geq 1}$ and a decreasing sequence $(u_\ell^+)_{\ell\geq 1}$ of $\mbox{supp }\mu_P$ such that $\lim_{\ell\rightarrow\infty}u_\ell^-= u_0=\lim_{\ell\rightarrow\infty}u_\ell^+.$ Then $\mu_P$ is continuous at $u_0.$ 

\end{enumerate}
\end{theorem}

\begin{proof}
 Take $x \in \mbox{supp } \mu_P$ and define $M= \max_{t \in [0,\hat{q}]} f(t).$ We claim that there exists a sequence of points $(x_n)_{n \geq 1} \subset S$ such that $(x_n)_{n \geq 1}$ converges to $x$. We argue by contradiction. If our claim does not hold, there exists an open neighborhood $O_x$ of $x$ such that $O_x \cap S = \emptyset$. However, since $x \in \mbox{supp } \mu_P$, $\mu_P(O_x)>0$ and this contradicts Proposition \ref{prop:characterizationSphere}. Now, since $f(x_n)=M$ and $f$ is continuous on $[0,1)$, we get $f(x)=M$ and therefore $x \in S.$ This proves $(i)$.
 
 Next, suppose $(a,b) \subset  \mbox{supp } \mu_P$. From item $(i)$, we have $(a,b) \subset S$.  From \eqref{eq:definitionfandF}, we see that $f$ is twice differentiable on $(0,1)$ with
\begin{equation*}\label{derivative}
f(0) = 0, \ f'(q)=F(q) \text{ and } f''(q)=\xi''(q)-\frac{1}{\hat{x}(q)^2}.
\end{equation*}
Hence, any $u \in (a,b) \subset S$ satisfies $f'(u)=0$ and consequently, 
\begin{equation}\label{eq:der2}
\xi''(u)^{-\frac{1}{2}} = \hat{x}_\mu(u).
\end{equation}
Since $\xi''(u)$ is positive and differentiable for any $u \in (a,b)$, a straightforward computation implies that the right derivative of $\hat{x}_\mu(u)$ is equal to $-\mu_P([0,u])$ and the left derivative is equal to $-\mu_P([0,u))$. By \eqref{eq:der2}, we obtain that $\mu_P([0,u]) = \mu_P([0,u))$ which means $\mu_P$ is continuous on $(a,b)$. Again from \eqref{eq:der2} and the fundamental theorem of calculus, we have
$$ \mu_P([0,u]) = - \hat {x}_\mu'(u)  = \frac{\xi'''(u)}{2\xi''(u)^{\frac{3}{2}}}.$$
This proves $(ii)$. 

Now suppose  $0 \in \mbox{supp }\mu_P$ and the existence of a sequence $u_n \downarrow 0$ such that $u_n \in  \mbox{supp }\mu_P$. Then by part $(i)$ and the mean value theorem, there exists a sequence $u_n' \downarrow 0$ such that $f''(u_n)=0$. By the continuity of $f''$ at $0$, we have $f''(0) = 0$. This immediately implies that $2\beta_2^2 = \xi''(0) =1$ giving item $(iii)$.

Next, to see that $(iv)$ holds, one argues similarly. The two sequences in $S$ converging to $u_0$ imply 
\begin{align*}
0=\lim_{h\rightarrow 0+}\frac{f''(u_0+h)-f''(u_0)}{h}&=\xi'''(u_0)-\frac{\mu_P([0,u_0])}{2\hat{x}_{\mu_P}(u_0)^3},  \\
0=\lim_{h\rightarrow 0-}\frac{f''(u_0+h)-f''(u_0)}{h}&=\xi'''(u_0)-\frac{\mu_P([0,u_0))}{2\hat{x}_{\mu_P}(u_0)^3} .
\end{align*} 
This gives us $\mu_P([0,u_0)) = \mu_P([0,u_0])$.

\end{proof}
\begin{example} [$(2+p)$-spin spherical model]  \rm Consider the case $$\xi(u) = \beta^2((1-t)u^2+tu^p)$$ for $t\in(0,1)$ and $p\geq 4$.
We claim that if  
\begin{equation}\label{eq:condmixed}
\frac{t}{1-t} \leq \frac{4(p-3)}{(p-1)p^2} \text{ and } \beta^2>\frac{1}{2(1-t)},
\end{equation}
then the model is FRSB with a jump at the top of the support. Furthermore, the Parisi measure $\mu_P$ is given by 
\begin{equation}\label{eq:measurea}
\mu_P([0,u]) =
\left\{
\begin{array}{ll}  
\frac{\xi'''(u)}{2\xi''(u)^{\frac{3}{2}}}, &\mbox{for } u<q_M,\\
1, &\mbox{for } u\geq q_M, 
\end{array}
\right.
\end{equation}
for some $q_M \in (0,1).$ 

We use Proposition  \ref{prop:characterizationSphere} to prove this claim. Indeed, it suffices to check  that $\mu_P(S)=1.$  Let $\phi(u)=\xi''(u)^{-1/2}$. Condition \eqref{eq:condmixed} implies that $\phi$ is concave, $\phi(0) <1$ and $\phi(1)>0$. Therefore, the graph of $\phi$ on $[0,1]$ intersects the line $y=1-x$ at a single point $q_M <1$. 
Since $\phi(q)>1-q$ for $q>q_M$, we have $\xi''(q) < 1/(1-q)^2$ for $q>q_M$. This implies that $F(q)<0$ for $q>q_M$. Now, \eqref{eq:measurea} implies $f(q) = 0$ for $q\leq q_M$ and $f(q) <0$ for $q>q_M$. Thus, $S=[0,q_M]$ and $\mu_P(S)=1$. A non-rigorous discussion of this model can be found in \cite{CriLeuzzi06}. 
\end{example}
\bigskip

% % % % % % % % % % % % % % % % % % % % % % % % % % % % % % % % % % % % % % % % % % % % % % % % % % % % % % % % % % % % % % % % % % % % % % % % % % % %
% % % % % % % % % % % % % % % % % % % % % % % % % % % % % % % % % % % % % % % % % % % % % % % % % % % % % % % % % % % % % % % % % % % % % % % % % % % %% % % % % % % % % % % % % % % % % % % % % % % % % % % % % % % % % % % % % % % % % % % % % % % % % % % % % % % % % % % % % % % % % % % % % % % % % % % %% % % % % % % % % % % % % % % % % % % % % % % % % % % % % % % % % % % % % % % % % % % % % % % % % % % % % % % % % % % % % % % % % % % % % % % % % % % %% % % % % % % % % % % % % % % % % % % % % % % % % % % % % % % % % % % % % % % % % % % % % % % % % % % % % % % % % % % % % % % % % % % % % % % % % % % %% % % % % % % % % % % % % % % % % % % % % % % % % % % % % % % % % % % % % % % % % % % % % % % % % % % % % % % % % % % % % % % % % % % % % % % % % % % %% % % % % % % % % % % % % % % % % % % % % % % % % % % % % % % % % % % % % % % % % % % % % % % % % % % % % % % % % % % % % % % % % % % % % % % % % % % %

\begin{flushleft}
\textbf{\Large Appendix}
\end{flushleft}

\noindent This appendix is devoted to proving Proposition \ref{thm1}. Let $\mu\in M[0,1].$ Recall the definition of $W_\mu$ from \eqref{eq:eqU}. It can also be written as
\begin{align}
\label{app:eq1}
W_{\mu}(u)&=\int_0^1\Phi_{\mu}(M(u),u)-\Phi_\mu(M(t\wedge u),t\wedge u)d\mu(t).
\end{align}
Define for $x\in\mathbb{R}$ and $u\in [0,1],$
\begin{align}
\label{app:eq2}
V_{\mu}(x,u)&=\int_0^1\Phi_{\mu}(x+M(1)-M(u),1)-\Phi_{\mu}(x+M(u\vee t)-M(u),u\vee t)d\mu(t).
\end{align}
As we have already mentioned in Section \ref{section:pde}, $\Phi_\mu$ is a continuous function on $\mathbb{R}\times[0,1].$ This implies that the preceding random functions are almost surely continuous.
Define $F_1(w)=w$ and for $j\geq 1,$
\begin{align}
\begin{split}\label{eq1}
F_{j+1}(y_1\ldots,y_j,w)
&=\sum_{i=1}^{j-1}\partial_{y_i}F_{j}(y_1,\ldots,y_{j-1},w)y_{i+1}\\
&+\partial_{w}F_{j}(y_1,\ldots,y_{j-1},w)(1-w^2)+F_j(y_1,\ldots,y_{j-1},w)y_1.
\end{split}
\end{align}
Let us start by summarizing some regularity properties of Parisi PDE solution when the measure consists of a finite number of atoms.

\begin{proposition}\label{lem1} Suppose that $\mu\in M_d[0,1]$ consists of a finite number of atoms on $(q_p)_{p=1}^{k+1}$ with $\mu([0,q_p])=m_{p}$ for $1\leq p\leq k+1.$ Then the following statements hold.
\begin{itemize}
\item[$(i)$] For $j\geq 0,$ 
\begin{align}
\label{app:lem1:eq2}
\begin{split}
\partial_x^j\Phi_\mu&\in C(\mathbb{R}\times[0,1]),\,\,\forall j\geq 0.
\end{split}
\end{align}
In particular, for $j\geq 1,$ 
\begin{align}
\begin{split}\label{lem1:eq1}
\partial_x^j\Phi_\mu(x,u)&=\e F_j(\partial_xV_\mu(x,u),\ldots,\partial_x^{j-1}V_\mu(x,u),\\
&\qquad\tanh(x+M(1)-M(u)))\exp V_\mu(x,u)
\end{split}
\end{align}
and
\begin{align}
\begin{split}
\label{lem1:eq2}
\sup_{\mathbb{R}\times [0,1]}|\partial_x^j\Phi_\mu|\leq C_{0,j}.
\end{split}
\end{align}

\item[$(ii)$] We also have
\begin{align}
\begin{split}\label{lem1:eq4}
\Phi_\mu|_{\mathbb{R}\times[q_{p-1},q_p)}&\in C^\infty(\mathbb{R}\times[q_{p-1},q_p)),\,\,\forall 1\leq p\leq k+1,\\
\Phi_\mu|_{\mathbb{R}\times[q_{k+1},1]}&\in C^\infty(\mathbb{R}\times[q_{k+1},1]).
\end{split}
\end{align}
For $j\geq 0,$
\begin{align}
\begin{split}\label{lem1:eq3}
\max \{\max_{1\leq p\leq k+1}\sup_{\mathbb{R}\times [q_{p-1},q_{p})}\left|\partial_u\partial_x^j\Phi_\mu|_{\mathbb{R}\times[q_{p-1},q_p)}\right|,&
\sup_{\mathbb{R}\times [q_{k+1},1]}\left|\partial_u\partial_x^j\Phi_\mu|_{\mathbb{R}\times[q_{k+1},1]}\right| \}\leq C_{1,j}.
\end{split}
\end{align}
\end{itemize}
Here $C_{i,j}$'s depend only on $\xi$.
\end{proposition}

\begin{proof} 
First let us prove $(\ref{lem1:eq1})$ and $(\ref{lem1:eq2})$ simultaneously by induction. The base case $j=1$ relies on Lemmas 3.3 and 3.4 in \cite{Tal07} which state respectively that $(\ref{lem1:eq1})$ holds for $j=1$ and that $\e\exp V_\mu=1$. These and $|\tanh|\leq 1$ imply \eqref{lem1:eq2} with $j=1$ and $C_{0,1}=1.$ Suppose that there exists some $j\geq 1$ such that $(\ref{lem1:eq1})$ and \eqref{lem1:eq2} hold for all $1\leq j'\leq j$. Note that $\tanh'=1-\tanh^2.$ From these, a direct differentiation leads to 
\begin{align*}
\partial_x^{j+1}\Phi_\mu(x,u)&=\e F_{j+1}(\partial_xV_\mu(x,u),\ldots, \partial_x^jV_\mu(x,u),\\
&\qquad \tanh(x+M(1)-M(u)))\exp V_\mu(x,u).
\end{align*}
Now observe that since $F_{j+1}$ is a polynomial and $\e \exp V_\mu=1,$ they together with the induction hypothesis give $(\ref{lem1:eq2})$ when $j$ is replaced by $j+1.$ This completes the proofs for $(\ref{lem1:eq1})$ and $(\ref{lem1:eq2}).$ Also an induction argument by using the definition \eqref{app:eq2} of $V_\mu$ and \eqref{lem1:eq1} yield \eqref{app:lem1:eq2}. This gives $(i).$

Next let us prove $(ii).$ Recall that $\Phi_\mu$ can be solved by \eqref{eq:ParPDE:eq1} and \eqref{eq:ParPDE:eq2}. By an induction argument using Gaussian integration by parts, \eqref{lem1:eq4} follows immediately.  As for $(\ref{lem1:eq3})$, note that $\Phi_\mu$ satisfies
$$
\partial_u\Phi_\mu(x,u)=-\frac{\xi''(u)}{2}\left(\partial_x^2\Phi_\mu(x,u)+m_{p-1}(\partial_x\Phi_\mu(x,u))^2\right)
$$
whenever $(x,u)\in \mathbb{R}\times[q_{p-1},q_{p})$ for $0\leq p\leq k+1$ and $(x,u)\in \mathbb{R}\times[q_{k+1},1].$ The inequalities \eqref{lem1:eq2} together with an application of the product rule imply
\begin{align*}
&\max \{ \max_{1\leq p\leq k+1}\sup_{\mathbb{R}\times[q_{p-1},q_{p})}|\partial_u\partial_{x}^{j}\Phi_\mu|,\sup_{\mathbb{R}\times[q_{k+1},1]}|\partial_u\partial_{x}^{j}\Phi_\mu| \}
\\
&\leq\frac{\xi''(1)}{2}\left(C_{0,j+2}+\sum_{j'=0}^{j}{j\choose j'}C_{0,j'+1}C_{0,j-j'+1}\right).
\end{align*}
This finishes our proof.
\end{proof}

\begin{lemma}\label{app:lem1}
There exists a constant $C$ depending only on $\xi$ such that for any $\mu\in M[0,1],$
\begin{align}
\begin{split}\label{app:lem1:eq1}
\sup_{\mathbb{R}\times[0,1]}\e\exp 3V_{\mu}\leq C,
\end{split}\\
\begin{split}\label{app:lem1:eq3}
\sup_{[0,1]}\e\exp 3W_{\mu}\leq C.
\end{split}
\end{align}
\end{lemma}

\begin{proof}
Since the arguments for both \eqref{app:lem1:eq1} and \eqref{app:lem1:eq3} are the same, we will only provide the detailed proof for \eqref{app:lem1:eq3}. First we claim that it holds when $\mu\in M_d[0,1]$. Observe that for $t\geq u,$
\begin{align*}
\Phi_\mu(M(u),u)-\Phi_\mu(M(t\wedge u),t\wedge u)=0
\end{align*}
and for $t<u,$ we may write
\begin{align*}
&\Phi_\mu(M(u),u)-\Phi_\mu(M(t\wedge u),t\wedge u)\\
&=\Phi_\mu(M(u),u)-\Phi_\mu(M(u),t\wedge u)\\
&+\Phi_\mu(M(u),t\wedge u)-\Phi_\mu(M(t\wedge u),t\wedge u)
\end{align*}
and apply the mean value theorem, \eqref{lem1:eq2} and \eqref{lem1:eq3} to get 
\begin{align*}
|\Phi_\mu(M(u),u)-\Phi_\mu(M(t\wedge u),t\wedge u)|&\leq C(|M(u)-M(t\wedge u)|+|u-t\wedge u|).
\end{align*}
Thus, 
\begin{align*}
|W_\mu(u)|&\leq C\sum_{p\geq 1:q_p\leq u}(m_p-m_{p-1})(|M(u)-M(q_p\wedge u)|+|u-t\wedge u|)\\
&\leq C\sum_{p=1}^{k+1}(m_p-m_{p-1})(|M(u)-M(q_p\wedge u)|+1).
\end{align*}
Note that for $1\leq p\leq k+1,$
\begin{align*}
\e \exp \pm 3C(M(u)-M(q_p\wedge u))=\exp \frac{9C^2}{2}|\xi'(u)-\xi'(q_p\wedge u)|\leq \exp \frac{9C^2}{2}\xi'(1).
\end{align*}
Since $\sum_{p=1}^{k+1}(m_p-m_{p-1})=1,$ using H\"{o}lder's inequality and the inequality $e^{|y|}\leq e^{y}+e^{-y}$ for all $y\in\mathbb{R}$ imply
\begin{align*}
\e \exp 3W_\mu(u)&\leq \exp 3C \prod_{p=1}^{k+1}(\e\exp 3C|M(u)-M(q_p\wedge u)|)^{m_p-m_{p-1}}\\
&\leq \exp 3C\prod_{p=1}^{k+1}\left(2\exp \frac{9C^2}{2}\xi'(1)\right)^{m_p-m_{p-1}}\\
&=2\exp \left(3C+\frac{9C^2}{2}\xi'(1)\right).
\end{align*}
For general $\mu,$ we approximate $\mu$ by a sequence of probability measures $(\mu_n)_{n\geq 1}\subset M[0,1]$. Since $\Phi_\mu(M(u),u)-\Phi_\mu(M(t\wedge u),t\wedge u)$ is a continuous function in $t\in[0,1],$ the weak convergence of $(\mu_n)_{n\geq 1}$ implies $\lim_{n\rightarrow\infty}W_{\mu_n}=W_\mu$ a.s. and the Fatou lemma concludes our assertion.
\end{proof}

Next, we will need a basic lemma in probability theory.

\begin{lemma}\label{lem3} Assume that $(\mu_n)_{n\geq 1}\subset M[0,1]$ converges weakly to $\mu$. For $n\geq 1,$ suppose that $\mathcal{F}_n$ is a collection of continuous functions on $[0,1]$ such that $\sup_{h\in\mathcal{F}_n}\|h\|_\infty\leq M$ and there are $\varepsilon,\delta>0$ such that $\sup_{h\in\mathcal{F}_n}|h(w)-h(w')|<\varepsilon$ whenever $|w-w'|<\delta.$ Then 
$$
\limsup_{n\rightarrow\infty}\sup_{h\in\mathcal{F}_n}\left|\int_0^1 hd\mu_n-\int_0^1 hd\mu\right|\leq 2\varepsilon.
$$
\end{lemma}

\begin{proof}
Let $F_n$ and $F$ be the distribution functions of $\mu_n$ and $\mu,$ respectively. Let us partition $[0,1]$ by intervals $I_1=[0,a_1],I_2=(a_1,a_2],\ldots,I_k=(a_{k-1},a_k]$ with $|I_j|<\delta$ for $j=1,2\ldots,k$. These $I_j$ may be chosen so that they are intervals of continuity of $F.$ Since each $h\in\mathcal{F}_n$ satisfies $\sup_{w,w'\in I_j}|h(w)-h(w')|\leq \varepsilon,$ we can approximate $h$ by a step function $g$ assuming constant value in each $I_j$ and such that $|h(w)-g(w)|<\varepsilon$ for all $w\in [0,1].$ Then
\begin{align*}
\left|\int_0^1 (h-g)d\mu\right|&\leq \varepsilon,\\
\left|\int_0^1 (h-g)d\mu_n\right|&\leq \varepsilon,\forall n\geq 1\ .
\end{align*}
Since $F_n(a_j)\rightarrow F(a_j)$ for each $j=1,2,\ldots,k-1$ and $F_{n}(1)=1\rightarrow F(1)=1$, we have that for sufficiently large $n,$
\begin{align*}
\left|\int_0^1 hd(\mu-\mu_n)\right|&\leq \left|\int_0^1 (h-g)d\mu\right|+\left|\int_0^1 gd(\mu-\mu_n)\right|+\left|\int_0^1 (g-h)d\mu_n\right|\\
&\leq 2\varepsilon+\left|\int_0^1 gd(\mu-\mu_n)\right|\\
&\leq 2\varepsilon+\sum_{j=0}^{k-1}|g(a_{j+1})|\left|(F_n(a_{j+1})-F_n(a_j))-(F(a_{j+1})-F(a_j))\right|\\
&\leq  2\varepsilon+M\sum_{j=0}^{k-1}\left(|F_n(a_{j+1})-F(a_{j+1})|+|F_{n}(a_{j})-F(a_j)|\right).
\end{align*}
Since this holds for any $h\in\mathcal{F}_n$, letting $n\rightarrow\infty$ gives our assertion.
\end{proof}

Using the preceding lemma, we have the following:

\begin{lemma}\label{lem5}
Suppose $(\mu_n)_{n\geq 1}\subset M_d[0,1]$ converges weakly to $\mu$.
\begin{itemize}
\item[$(i)$] If $(\partial_x^j\Phi_{\mu_n})_{n\geq 1}$ converges to $\partial_x^j\Phi_\mu$ uniformly on $\mathbb{R}\times [0,1]$ for $j\geq 0$, then
\begin{align}\label{lem5:eq1}
\lim_{n\rightarrow\infty}\sup_{\in\mathbb{R}\times[0,1]}\e|\partial_x^jV_{\mu_n}- \partial_x^jV_\mu|^3=0.
\end{align}
\item[$(ii)$] We have
\begin{align}
\label{lem5:eq2}
\lim_{n\rightarrow\infty}\sup_{[0,1]}\e|W_{\mu_n}-W_\mu|^3=0.
\end{align}
\end{itemize}

\end{lemma}

\begin{proof} Again, since crucial arguments for both \eqref{lem5:eq1} and \eqref{lem5:eq2} are essentially the same, we will only prove \eqref{lem5:eq2}. For $f\in C[0,1]$ and $(x,u)\in\mathbb{R}\times[0,1],$ we define 
\begin{align*}
A_{n}^{f,u}(t)&=\Phi_{\mu_n}(f(u),u)-\Phi_{\mu_n}(f(t\wedge u),t\wedge u),\,\,\forall n\geq 1,\\
A^{f,u}(t)&=\Phi_\mu(f(u),u)-\Phi_\mu(f(t\wedge u),t\wedge u).
\end{align*}
For $\varepsilon>0$ and $f_0\in C[0,1],$ set an open ball $$
B(f_0,\varepsilon)=\{f\in C[0,1]:\|f-f_0\|_\infty<\varepsilon\}$$ 
and for $n\geq 1,$
$$
\mathcal{F}_n(f_0)=\{A_{n}^{f,u}:u\in [0,1],f\in B(f_0,\varepsilon)\}\subset C[0,1].
$$
For arbitrary $0\leq t,t'\leq 1,$ write
\begin{align*}
A_{n}^{f,u}(t)-A_{n}^{f,u}(t')&=\Phi_{\mu_n}(f(t\wedge u),t\wedge u)-\Phi_{\mu_n}(f(t'\wedge u),t'\wedge u)\\
&=\Phi_{\mu_n}(f(t\wedge u),t\wedge u)-\Phi_{\mu_n}(f(t\wedge u),t'\wedge u)\\
&+\Phi_{\mu_n}(f(t\wedge u),t'\wedge u)-\Phi_{\mu_n}(f(t'\wedge u),t'\wedge u).
\end{align*}
Using \eqref{lem1:eq2} and \eqref{lem1:eq3}, the mean value theorem implies
\begin{align*}
|A_{n}^{f,u}(t)-A_{n}^{f,u}(t')|&\leq C(|t\wedge u-t'\wedge u|+|f(t\wedge u)-f(t'\wedge u)|).
\end{align*}
This implies two immediate consequences. First there exists some $\delta>0$ such that
\begin{align}
\label{add:lem3:proof:eq1}
\sup_{h \in \mathcal{F}_n(f_0)}|h(t)-h(t')|\leq C\varepsilon
\end{align}
whenever $|t-t'|<\delta$. Second, since $h(u)=0$ for all $h\in\mathcal{F}_0(f_0),$ 
\begin{align}
\label{add:lem3:proof:eq2}
\sup_{h \in \mathcal{F}_n(f_0)}\|h\|_\infty<\infty.
\end{align} 
From $(\ref{add:lem3:proof:eq1})$ and $(\ref{add:lem3:proof:eq2})$, applying Lemma $\ref{lem3}$ to $\mathcal{F}_n(f_0)$ yields 
\begin{align}
\label{add:lem3:proof:eq3}
\limsup_{n\rightarrow\infty}\sup_{\mathcal{F}_n(f_0)}\left|\int_0^1 hd(\mu_n-\mu)\right|\leq 2C\varepsilon.
\end{align}

\smallskip

Next using the tightness of the time changed Brownain motion $(M(u))_{0\leq u\leq 1}$, there exists a compact set $\mathcal{K}\subset C([0,1])$ such that $\p(\mathcal{K}^c)<\varepsilon.$ Let us cover $\mathcal{K}$ by a finite number of open balls $B(f_j,\varepsilon)$ for $j=1,2,\ldots,\ell.$ 
Write
\begin{align}
\begin{split}
\label{add:lem3:proof:eq6}
W_{\mu}(u)-W_{\mu_n}(x,u)&=\int_0^1A^{M,u}(t)d\mu(t)-\int_0^1A_{n}^{M,u}(t)d\mu_n(t)\\
&=\int_0^1\left(A^{M,u}(t)-A_n^{M,u}(t)\right)d\mu(t)+\int_0^1A_n^{M,u}(t)d(\mu-\mu_n)(t).
\end{split}
\end{align}
Recall that as we have mentioned before, $(\Phi_{n})_{n\geq 1}$ converges to $\Phi$ uniformly on $\mathbb{R}\times[0,1].$ It follows that 
\begin{align}
\label{add:lem3:proof:eq7}
\sup_{f\in C[0,1]}\|A^{f,u}-A_n^{f,u}\|_\infty<\varepsilon
\end{align}
for sufficiently large $n$. From $(\ref{add:lem3:proof:eq6})$ and $(\ref{add:lem3:proof:eq7})$, it follows that
\begin{align*}
&\sup_{[0,1]}\e\left(|W_\mu(u)-W_{\mu_n}(u)|^3:\mathcal{K}\right)\\
&\leq 2^3\varepsilon^3+2^3\sup_{[0,1]}\sum_{i=1}^\ell\e \left(\left|\int_0^1A_n^{M,u}d(\mu-\mu_n)\right|^3:B(f_i,\varepsilon)\cap \mathcal{K}\right)\\
&\leq 2^3\varepsilon^3+2^3\sum_{i=1}^\ell\sup_{h\in \mathcal{F}_n(f_i)}\left|\int_0^1 hd(\mu-\mu_n)\right|^3\p\left(B(f_i,\varepsilon)\cap \mathcal{K}\right)
\end{align*}
and then from $(\ref{add:lem3:proof:eq3})$,
\begin{align}
\begin{split}
\label{add:lem3:proof:eq8}
\limsup_{n\rightarrow\infty}\sup_{[0,1]}\e\left(|W_\mu(u)-W_{\mu_n}(u)|^3:\mathcal{K}\right)&\leq 2^3\varepsilon^3+2^6C^3\varepsilon^3\sum_{i=1}^\ell\p(B(f_i,\varepsilon)\cap \mathcal{K})\\
&= 2^3\varepsilon^3+2^6C^3\varepsilon^3\p(\mathcal{K})\\
&\leq (2^3+2^6C^2)\varepsilon^2.
\end{split}
\end{align}
On the other hand, using Lemma $\ref{lem1}$ again, we have
\begin{align*}
&\sup_{[0,1]}\e\left(|W_{\mu}(u)-W_{\mu_n}(u)|^3:\mathcal{K}^c\right)\leq C\p(\mathcal{K}^c)\leq C\varepsilon.
\end{align*}
This together with $(\ref{add:lem3:proof:eq8})$ completes our proof.
\end{proof}

% % % % % % % % % % % % % % % % % % % % % % % % % % % % % % % % % % % % % % % % % % % % % % % % % % % % % % % % % % % % % % % % % % % % % % % % % % % % % % % % % % % % % % % % % % % % % % % % % % % % % % % % % % % % % % % % % % % % % % % % % % % % % % % % % % % % % % % % % % % % % % % % % % % % % % % % % % % % % % % % % % % % % % % % % % % % % % % % % % % % % % % % % % % % % % % % % % % % % % % % % % % % % % % % % % % % % % % % % % % % % % % % % % % % % % %

\begin{lemma}\label{lem6} Let $(\mu_n)_{n\geq 1}\subset M_d[0,1]$ converge weakly to $\mu$. For $j\geq 1,$ let $$
Q(\mathbf{y})=P(y_1,\ldots,y_j)\exp y_0$$
for $\mathbf{y}=(y_0,y_1,\ldots,y_j)\in\mathbb{R}^{j+1}$, where $P$ is a polynomial. Suppose that $\partial_x^{j'}\Phi_{\mu}$ exists and $\lim_{n\rightarrow\infty}\partial_x^{j'}\Phi_{\mu_n}=\partial_x^{j'}\Phi_\mu$ uniformly on $\mathbb{R}\times[0,1]$ for $0\leq j'\leq j.$ Then
\begin{align}\label{lem6:eq1}
\lim_{n\rightarrow\infty}\sup_{\mathbb{R}\times[0,1]}\e |Q(V_{\mu_n},\partial_xV_{\mu_n},\ldots,\partial_x^jV_{\mu_n})- Q(V_\mu,\partial_xV_\mu,\ldots,\partial_x^jV_\mu)|=0
\end{align}
and
\begin{align}
\label{lem6:eq3}
\lim_{n\rightarrow\infty}\sup_{[0,1]}\e |Q(W_{\mu_n},W_{\mu_n}^1,\ldots,W^j_{\mu_n})- Q(W_\mu,W^1_\mu,\ldots,W^j_\mu)|=0,
\end{align}
where for any $\nu\in M[0,1]$, $W_\nu^{j'}(u):=\partial_x\Phi_{\nu}^{j'}(M(u),u)$ for $0\leq u\leq 1$ and $1\leq j'\leq j.$
\end{lemma}

\begin{proof} Once again we will only provide the detailed proof for \eqref{lem6:eq3}. As for \eqref{lem6:eq1}, it can be treated by a similar argument as \eqref{lem6:eq1}. Let $C_{0,1},\ldots,C_{0,j}$ be the bounds we obtained from \eqref{lem1:eq2}. 
Using the mean value theorem, there exists some constant $C$ depending only on $C_{0,1},\ldots,C_{0,j}$ such that
\begin{align*}
|Q(\mathbf{y})-Q(\mathbf{y}')|&=|\triangledown Q(t\mathbf{y}+(1-t)\mathbf{y}')\cdot (\mathbf{y}-\mathbf{y}')|\\
&\leq \|\triangledown Q(t\mathbf{y}+(1-t)\mathbf{y}')\|\|\mathbf{y}-\mathbf{y}'\|\\
&\leq C\exp (|ty_0+(1-t)y_0'|)\|\mathbf{y}-\mathbf{y}'\| \\
&\leq C\exp( |y_0|+|y_0'|)\sum_{i=0}^{j}|y_i-y_i'|
\end{align*}
whenever $\mathbf{y},\mathbf{y}'$ satisfy $|y_i|,|y_i'|\leq C_{0,i}$ for $1\leq i\leq j.$ 
From this and H\"{o}lder's inequality, it suffices to prove that for some $C>0$,
\begin{align}\label{add:lem6:proof:eq1}
\max\{ \ \sup_{[0,1]}\e\exp 3|W_{\mu_n}|,\,\,\sup_{[0,1]}\e\exp 3|W_\mu|\ \}&\leq C
\end{align}
and
\begin{align}
\begin{split}\label{add:lem6:proof:eq2}
\lim_{n\rightarrow\infty}\sup_{[0,1]}\e |W_{\mu_n}-W_\mu|^3&=0,
\end{split}\\
\begin{split}\label{add:lem6:proof:eq3}
\lim_{n\rightarrow\infty}\sup_{[0,1]}\e |W_{\mu_n}^i-W^i_\mu|^3&=0,\,\,1\leq i\leq j.
\end{split}
\end{align}
Here \eqref{add:lem6:proof:eq1} and \eqref{add:lem6:proof:eq2} follow respectively from \eqref{app:lem1:eq3} and \eqref{lem5:eq2}, while \eqref{add:lem6:proof:eq3} holds directly from the uniform convergence of $(\partial_x^i\Phi_{\mu_n})_{n\geq 1}$. This completes our proof.
\end{proof}

\begin{proof}[Proof of Proposition \ref{thm1}] Note that since $M_d[0,1]$ is a dense subset of $M[0,1]$, it suffices to consider $(\mu_n)_{n\geq 1}\subset M_d[0,1]$ with a weak limit $\mu.$ First, let us prove $(i)$ by induction. The base case follows by the Lipschitz property of the functional $\Phi_{\cdot}$ as discussed in Section \ref{section:pde}. Suppose that there exists some $j\geq 0$ such that the announced result holds for each $0\leq j'\leq j.$ Recall $F_{j+1}$ from $(\ref{eq1})$. Observe that it can be written as
\begin{align}\label{thm1:proof:eq1}
F_{j+1}(y_1,\ldots,y_j,w)&=\sum_{i=0}^\ell P_i(y_1,\ldots,y_j)w^i,
\end{align}
where $P_0,\ldots,P_\ell$ are polynomials and $\ell\geq 1.$ Set $Q_{i}(y_0,y_1,\ldots,y_j)=P_i(y_1,\ldots,y_j)\exp y_0.$ Since  $\partial_x^{j'}\Phi_{\mu_n}$ is continuous from \eqref{app:lem1:eq2} and $\lim_{n\rightarrow\infty}\partial_x^{j'}\Phi_{\mu_n}=\partial_x^{j'}\Phi_\mu$ uniformly on $\mathbb{R}\times[0,1]$ for all $0\leq j'\leq j,$ applying $(\ref{lem6:eq1})$ and $|\tanh|\leq 1$ we obtain
\begin{align*}
&\limsup_{n\rightarrow\infty}\sup_{\mathbb{R}\times[0,1]}\e |Q_i(V_{\mu_n}(x,u),\partial_xV_{\mu_n}(x,u),\ldots,\partial_{x}^jV_{\mu_n}(x,u))\\
&\qquad\qquad-Q_i(V_\mu(x,u),\partial_xV_\mu(x,u),\ldots,\partial_{x}^jV_\mu(x,u))||\tanh^i(x+M(1)-M(u))|=0.
\end{align*} 
From  \eqref{thm1:proof:eq1}, this and \eqref{lem1:eq1} imply that $\{\partial_x^{j+1}\Phi_{\mu_n}\}_{n\geq 1}$ converges uniformly on $\mathbb{R}\times[0,1]$. So $\partial_x^{j+1}\Phi_{\mu}=\lim_{n\rightarrow\infty}\partial_x^{j+1}\Phi_{\mu_n}$ exists and is continuous. This gives $(i).$

Next let us turn to the proofs of $(ii)$ and $(iii).$ Define $Q(\mathbf{y})=P(y_1,\ldots,y_j)\exp y_0$ for $\mathbf{y}=(y_0,y_1,\ldots,y_j).$ From part $(i)$, we know that $(\partial_x^{j'}\Phi_{\mu_n})_{n\geq 1}$ converges uniformly to $\partial_x^{j'}\Phi_\mu$ on $\mathbb{R}\times[0,1]$ for all $0\leq j'\leq j.$ Consequently, $(ii)$ follows from \eqref{lem6:eq3}. As for $(iii)$, it can be simply concluded from $(i)$ and the definition of the Parisi PDE \eqref{eq:ParPDE}.
\end{proof}

\end{document}